\documentclass[aap,press]{apt} \volno{00} \partno{00}
\firstpage{1}
\usepackage[colorlinks = true,linkcolor = blue,urlcolor  = blue,citecolor = blue,anchorcolor = blue]{hyperref}

\paperno{XXXX-XXXX} 
\authornames{C. KRESIN {\it et al.}} 
\shorttitle{Bispectra of Poisson cluster processes} 

\usepackage{amssymb,mathtools,microtype}

\newcommand{\E}{\mathbb{E}}
\newcommand{\PP}{\mathbb{P}}
\newcommand{\R}{\mathbb{R}}
\renewcommand{\ii}{\mathrm{i}}
\newcommand{\1}{\mathbf{1}}
\newcommand{\Law}{\mathcal{L}}
\newcommand{\Ree}{\mathrm{Re}}
\newcommand{\Imm}{\mathrm{Im}}

\numberwithin{equation}{section}

\begin{document}

\title{Orientation in Poisson Cluster Processes via Imaginary Bispectra}

\authoronetwo[University of Otago]{Conor Kresin}
\authoronethree[University of Otago]{Yifu Tang}
\authoronefour[University of Otago]{Boris Baeumer}
\authoronefive[University of Otago]{Ting Wang}

\addressone{Department of Mathematics and Statistics, University of Otago, 730 Cumberland Street, Dunedin 9016, Otago, New Zealand}
\emailtwo{conor.kresin@otago.ac.nz}
\emailthree{yifu.tang@otago.ac.nz}
\emailfour{boris.baeumer@otago.ac.nz}
\emailfive{ting.wang@otago.ac.nz}

\begin{abstract}
We study what remains detectable about one-sided Poisson cluster processes after cluster orientation is erased. We construct matched reversible cluster nulls preserving intensity and the full Bartlett spectrum, showing that second-order structure alone need not identify temporal direction. For stationary Poisson branching clusters, we derive the Fourier--Stieltjes transform of the reduced third cumulant and show that, in the $L^1$ third-cumulant regime, a nonzero imaginary factorial bispectrum certifies orientation. We also give explicit orientation-erased nulls, reversible spectral matches for monotone Hawkes kernels, and finite-window third-order orientation contrasts.
\end{abstract}



\section{Introduction}

A clustered event sequence can be viewed in two related ways. In a temporal interpretation, later events are generated from earlier ones, as in a one-sided Hawkes or branching model. In a spatial interpretation, the same events are treated as an unoriented point pattern on the line, generated by latent clusters but without assigning causal meaning to the order of points inside a cluster. First- and second-order summaries largely belong to this spatial view: they measure intensity, pair distances, covariance, spectra, and overdispersion, but they do not by themselves determine a direction of generation.

We use the following matched-comparison framework. Let \(N_1\)
denote the oriented one-sided cluster process of interest, for example
a Hawkes branching-cluster process. A \emph{matched reversible
comparison} for \(N_1\) is a stationary cluster process \(\widetilde N\)
such that \(\widetilde N\stackrel d=R\widetilde N\), where
\((RN)(A):=N(-A)\), and such that \(\widetilde N\) has the same
intensity and the same complete reduced covariance measure as \(N_1\),
equivalently the same full Bartlett spectrum. Thus ``matched'' refers
only to first- and second-order structure: the comparison process is
not required to have the same full law as \(N_1\), and in particular
need not be a causal Hawkes process.

One concrete way to erase orientation is to reflect whole clusters:
replace a cluster \(C\) by \(SC\), where \(S=\pm1\) is an independent
sign. With a fair sign this produces a reversible cluster law. Because
ordered pair-difference measures are unchanged by whole-cluster
reflection, this orientation-erased process has the same intensity and
Bartlett spectrum as the forward process. In this precise sense,
second-order summaries can certify clustering but not temporal
orientation within the matched comparison class. The problem we address is therefore to identify what distinguishes the
oriented law \(N_1\) from such matched reversible comparisons, and to
find the lowest-order statistics that certify temporal orientation
against them.

The main results are as follows. First, the Fourier--Stieltjes transform of the reduced complete, or ordinary, third cumulant measure of a stationary Poisson branching-cluster process has a simple formula; no causal or one-sided support assumption enters the derivation. For a centrally symmetric offspring density on the line the process is reversible, and the  Fourier--Stieltjes  transforms of the reduced complete and factorial third cumulant measures are real. Second, per-cluster orientation randomization produces an explicit matched reversible null that preserves first and second order exactly. Third, if the reduced third factorial cumulant has an $L^1(\R^2)$ density, then reversibility of any comparison null in this sense forces the factorial bispectrum to be real, so a nonzero imaginary part gives a third-order certificate of oriented temporal structure. Fourth, for the per-cluster sign-biased family \(N_\theta\)
connecting the forward law \(N_1\) to its reversed orientation
\(N_{-1}\), let \(B_\theta\) denote the factorial bispectrum of
\(N_\theta\), whenever the reduced third factorial cumulant has an
\(L^1\) density. Then only the odd third-order component changes:
\[
B_\theta=\Ree B_1+\ii\theta\,\Imm B_1,
\]
where \(B_1\) is the bispectrum of the forward process.
Fifth, every monotone branching kernel admits an explicit reversible branching-cluster spectral match within the branching framework; this includes the exponential case and heavy-tailed monotone kernels. Finally, for broad classes of one-sided kernels, the imaginary bispectrum is not merely allowed by the theory but is provably nonzero somewhere. We also note one observable consequence: jointly odd third-order statistics have exact finite-window mean formulae linear in the sign parameter and large-window mean limits given by weighted imaginary-bispectrum functionals.

Linear Hawkes processes are canonical models for directed temporal
clustering and have a branching representation \cite{HawkesOakes1974}.
Their spectra are studied in \cite{Hawkes1971}. Spectral estimation is developed for count data in \cite{cheysson2022spectral} and for multivariate stationary Hawkes processes in \cite{tang2026spectral}; \cite{cheysson2022spectral} also discusses a non-causal cluster-based extension. This does not conflict with identifiability results inside the causal Hawkes class. Under the usual
stationarity and stability assumptions, and when the full matrix-valued
second-order structure is known, a causal multivariate linear Hawkes
kernel can be identified within that model class through a Wiener--Hopf
equation \cite{BacryMuzy2016}. Our comparison is deliberately outside
that class: even if \(N_1\) is Hawkes, a matched reversible comparison
is required only to reproduce its intensity and second-order structure.
It need not preserve the forward Hawkes law, the causal genealogy, or
the adapted conditional-intensity interpretation.

The conceptual consequence is that second-order agreement with a
Hawkes model can certify clustering without, by itself, certifying a
directed self-excitation mechanism in this broader comparison class. In the present matched-comparison framework, a nonzero odd third-order component is a sufficient third-order certificate of orientation: it is carried by the part of the reduced third factorial cumulant that changes sign under simultaneous reversal of the two lags, or equivalently by the imaginary part of the factorial bispectrum. This distinction is especially important in applications where Hawkes models are interpreted mechanistically rather than merely used for prediction.

\subsection{Relation to the literature}

Cumulants and spectra for stationary point processes go back at least to the classical framework of Brillinger \cite{Brillinger1975}. Higher-order spectra and cumulants for cluster and Hawkes processes are developed in \cite{TakahashiEtAl1998}, \cite{JovanovicHertzRotter2015}, and \cite{Privault2021}, and integrated second- and third-order cumulants have been used for multivariate Hawkes-network recovery \cite{AchabEtAl2018}. Spatial cluster process models present their own identifiability and fitting issues \cite{BaddeleyEtAl2022}. Recent work also uses cluster representations to derive Palm and spectral-intensity identities \cite{kirchner2026palm} and extends Bartlett identities to spatial point processes via variational derivatives of the log-likelihood \cite{clark2026bartlett}. These developments motivate asking which aspects of a causal one-sided cluster law are not already encoded in intensity and second-order structure.

Frequency-domain uses of the imaginary part of the ordinary bispectrum
as a witness of time irreversibility are classical in time-series
analysis: Hinich and Rothman \cite{HinichRothman1998} developed the
REVERSE test for scalar stationary time series by exploiting the
implication that reversibility forces the imaginary ordinary bispectrum
to vanish. Related work on real-valued bispectra \cite{IgloiTerdik2014} provides
a useful caution: the implication used here is one-way. Reversibility
forces the relevant bispectrum to be real under the stated assumptions,
but real-valuedness alone does not establish reversibility.

Our matched
reversible-null framing is different: we work with stationary point
processes and factorial cumulants, construct reversible cluster
comparisons preserving intensity and the full Bartlett spectrum, and
isolate the odd third-order factorial component as the remaining
orientation signal for Poisson cluster and Hawkes-type branching laws.
Thus, in the $L^1$ third-cumulant regime, reversibility of the matched
comparison forces a real factorial bispectrum, so a nonzero imaginary
factorial bispectrum gives a sufficient third-order certificate of
temporal orientation beyond matched first- and second-order structure.
This viewpoint also connects to Hawkes-specific time-reversal
diagnostics \cite{CordiChalletMuniToke2018}, to Granger interpretations
of multivariate Hawkes models \cite{EichlerDahlhausDueck2017}, and to
stochastic declustering and probabilistic reconstruction of triggering
structure from event times \cite{ZhuangOgataVereJones2002}.

\section{Frequency-domain third cumulants for branching clusters and reversible nulls}

This section collects the conventions and comparison-null facts used throughout the paper, separating two ideas. First, reversibility is the structural property of our comparison nulls and the reason their reduced third factorial cumulants are centrally symmetric; in the $L^1$ regime this yields a real factorial bispectrum. Second, the branching-cluster third-order Fourier--Stieltjes formula itself is not intrinsically causal: in the one-dimensional form needed here, it applies both to one-sided temporal Hawkes models and to centrally symmetric reversible branching nulls.

For $u\in\R$ and a counting measure $\mu$, let $\vartheta_u\mu$ denote the translate
\[
(\vartheta_u\mu)(A):=\mu(A-u),\qquad A\subset\R\ \text{Borel}.
\]
For a bounded compactly supported Borel function $f$, write
\[
N(f):=\int f\,dN = \sum_{x\in N}f(x) .
\]
For a bounded compactly supported Borel function $q$ with $1+q\ge0$, define the factorial moment generating functional by
\[
G_{\mathrm{fac}}(q)
:=
\E\prod_{x\in N}(1+q(x))
\]
whenever the expectation is finite. On the domain where this expectation is finite and positive, define the factorial cumulant generating functional by
\[
K_{\mathrm{fac}}(q)
:=
\log G_{\mathrm{fac}}(q).
\]
The complete, or ordinary, cumulant generating functional is, whenever finite,
\[
K_{\mathrm{comp}}(f)
:=
\log \E e^{N(f)} .
\]
Without local exponential-moment assumptions these functionals need not be finite for positive test functions. We use them only either on domains where the logarithms are finite, for example for nonpositive compactly supported $f$ and for $q\in(-1,0]$, or through their finite-order derivatives at the origin. Under the moment assumptions stated below, those derivatives are justified directly by truncation and dominated convergence.
Since
\[
e^{N(f)}=\prod_{x\in N} e^{f(x)}=\prod_{x\in N}\{1+(e^{f(x)}-1)\},
\]
we have, whenever both sides are finite,
\begin{equation}\label{eq:ordinary-factorial-generating-functional}
K_{\mathrm{comp}}(f)
=
K_{\mathrm{fac}}(e^f-1).
\end{equation}
At the level of derivatives at the origin this identity is the source of the usual diagonal corrections relating complete and factorial cumulants.

For $n\ge1$ and test functions $f_0,\ldots,f_{n-1}$, the $n$th factorial moment measure $\alpha_n$, factorial cumulant measure $\kappa_n$, and complete cumulant measure $\gamma_n$ are characterized, whenever the displayed functionals are differentiable at the origin, by the variational identities below. Under the moment hypotheses used later, the same formulae are interpreted directly as finite variational derivatives at the origin. Here and in the next two displays,
\(s=(s_0,\ldots,s_{n-1})\in\R^n\), the variables \(s_j\) are
independent scalar parameters, and \(\big|_{s=0}\) means evaluation at
\(s_0=\cdots=s_{n-1}=0\):
\[
\int_{\R^n}\prod_{j=0}^{n-1}f_j(t_j)\,
\alpha_n(dt_0,\ldots,dt_{n-1})
=
\left.
\partial_{s_0}\cdots\partial_{s_{n-1}}
G_{\mathrm{fac}}\!\left(\sum_{j=0}^{n-1}s_jf_j\right)
\right|_{s=0},
\]
\[
\int_{\R^n}\prod_{j=0}^{n-1}f_j(t_j)\,
\kappa_n(dt_0,\ldots,dt_{n-1})
=
\left.
\partial_{s_0}\cdots\partial_{s_{n-1}}
K_{\mathrm{fac}}\!\left(\sum_{j=0}^{n-1}s_jf_j\right)
\right|_{s=0},
\]
and
\[
\int_{\R^n}\prod_{j=0}^{n-1}f_j(t_j)\,
\gamma_n(dt_0,\ldots,dt_{n-1})
=
\left.
\partial_{s_0}\cdots\partial_{s_{n-1}}
K_{\mathrm{comp}}\!\left(\sum_{j=0}^{n-1}s_jf_j\right)
\right|_{s=0}.
\]
Here ``complete'' and ``ordinary'' are used synonymously. Unless explicitly stated otherwise, $\kappa_n$ denotes the factorial cumulant measure.

For $n\ge2$, the reduced versions of these measures are defined by stationarity through
\[
\int_{\R^n} F(t_0,\dots,t_{n-1})\,M_n(dt_0,\dots,dt_{n-1})
=
\int_\R\int_{\R^{n-1}}
F(t_0,t_0+\tau_1,\dots,t_0+\tau_{n-1})\,
M_n^{\mathrm{red}}(d\tau_1,\dots,d\tau_{n-1})\,dt_0,
\]
where $M_n$ is any of $\alpha_n,\kappa_n,\gamma_n$, whenever the corresponding reduced measure is well defined. These are the standard stationary point-process factorial moment and cumulant conventions used in \cite{Brillinger1975} and \cite{daley2003introduction}; the only normalization choice to note is that we do not factor out explicit powers of the intensity from these reduced measures.

For $q\in\{1,2\}$ and $f\in L^1(\R^q)$, we use the Fourier convention
\[
\widehat f(\omega):=\int_{\R^q} e^{-\ii\,\omega\cdot t}f(t)\,dt,
\qquad
f(t)=\frac{1}{(2\pi)^q}\int_{\R^q} e^{\ii\,\omega\cdot t}\widehat f(\omega)\,d\omega
\]
whenever $\widehat f\in L^1(\R^q)$. The same forward-transform convention is used for finite signed measures on $\R^q$.

When $\tau=(\tau_1,\tau_2)$, we write $-\tau:=(-\tau_1,-\tau_2)$. Expectations under the law of $N_\theta$ are denoted by $\E_\theta$. Throughout, $\sum^{\neq}$ denotes summation over ordered tuples of pairwise distinct points or vertices; equivalently, if a finite counting measure is represented by labelled atoms, the labels in the tuple are distinct. For an integer-valued random variable \(M\), we write
\[
(M)_3:=M(M-1)(M-2).
\]

With this convention,
\[
\gamma_2^{\mathrm{red}}=\lambda\delta_0+\kappa_2^{\mathrm{red}},
\]
and hence
\[
\Gamma(\omega):=\widehat{\gamma_2^{\mathrm{red}}}(\omega)
=\lambda+\int_\R e^{-\ii\omega\tau}\,\kappa_2^{\mathrm{red}}(d\tau)
\]
is the full Bartlett spectrum, meaning the Fourier transform of the complete reduced covariance measure including the Poisson diagonal baseline. Whenever this Fourier--Stieltjes transform is an ordinary function, the symmetry $\kappa_2^{\mathrm{red}}(A)=\kappa_2^{\mathrm{red}}(-A)$ makes $\Gamma$ real-valued and even.

We reserve the term \emph{bispectrum} for the case where the reduced third factorial cumulant has an $L^1$ density; otherwise we speak of the corresponding Fourier--Stieltjes transform of the reduced third factorial cumulant.

\begin{lem}[Cumulant functionals for Poisson cluster processes]\label{lem:pcp-cumulant-functional}
Let
\[
N=\sum_j\vartheta_{U_j}C_j
\]
be a stationary Poisson cluster process with immigrant rate $\nu$, where the $C_j$ are i.i.d. copies of a finite simple cluster $C$ and $\E C(\R)<\infty$; if multiplicities are allowed, products and sums below are over atoms counted with multiplicity. Whenever the following logarithms are finite, for instance for bounded compactly supported $q$ with $-1<q\le0$ and bounded compactly supported $f\le0$ under the preceding finite-mean assumption, the Poisson product formula gives
\[
K_{\mathrm{fac}}^N(q)
=
\nu\int_\R
\left\{
\E\prod_{y\in C}(1+q(u+y))-1
\right\}\,du
\]
and
\[
K_{\mathrm{comp}}^N(f)
=
\nu\int_\R
\left\{
\E\exp\big(C(f(u+\cdot))\big)-1
\right\}\,du .
\]
The derivative identities used below require only finite moments: for every $n\ge1$ with $\E[C(\R)^n]<\infty$ and bounded compactly supported $f_0,\ldots,f_{n-1}$,
\[
\int_{\R^n}\prod_{j=0}^{n-1} f_j(t_j)\,
\kappa_n(dt_0,\ldots,dt_{n-1})
=
\nu\int_\R
\E\sum_{x_0,\ldots,x_{n-1}\in C}^{\neq}
\prod_{j=0}^{n-1}f_j(u+x_j)\,du,
\]
whereas
\[
\int_{\R^n}\prod_{j=0}^{n-1} f_j(t_j)\,
\gamma_n(dt_0,\ldots,dt_{n-1})
=
\nu\int_\R
\E\sum_{x_0,\ldots,x_{n-1}\in C}
\prod_{j=0}^{n-1}f_j(u+x_j)\,du.
\]
Equivalently, for $n\ge2$,
\[
\kappa_n^{\mathrm{red}}(d\tau_1,\ldots,d\tau_{n-1})
=
\nu\,\E\sum_{x_0,\ldots,x_{n-1}\in C}^{\neq}
\prod_{r=1}^{n-1}\delta_{x_r-x_0}(d\tau_r),
\]
and
\[
\gamma_n^{\mathrm{red}}(d\tau_1,\ldots,d\tau_{n-1})
=
\nu\,\E\sum_{x_0,\ldots,x_{n-1}\in C}
\prod_{r=1}^{n-1}\delta_{x_r-x_0}(d\tau_r).
\]
In particular, if $\E[C(\R)^3]<\infty$ and
\[
W(\omega):=\widehat C(\omega)=\sum_{x\in C}e^{-\ii\omega x},
\]
then
\[
\widehat{\gamma_3^{\mathrm{red}}}(\omega_1,\omega_2)
=
\nu\,\E\{W(\omega_1)W(\omega_2)W(-\omega_1-\omega_2)\}.
\]
\end{lem}

\subsection{Reversibility and complete/factorial bispectra}

For the next result, assume that
\[
\kappa_3^{\mathrm{red}}(d\tau_1,d\tau_2)=c_3(\tau_1,\tau_2)\,d\tau_1d\tau_2,
\qquad c_3\in L^1(\R^2),
\]
and define the factorial bispectrum
\[
B_{\mathrm{fac}}(\omega_1,\omega_2)
:=
\iint_{\R^2} c_3(\tau_1,\tau_2)e^{-\ii(\omega_1\tau_1+\omega_2\tau_2)}\,d\tau_1d\tau_2.
\]
When there is no risk of confusion we write simply $B$ for $B_{\mathrm{fac}}$.

The next result is the factorial point-process analogue of the familiar
time-series implication that reversibility forces the ordinary
bispectrum  to be real-valued
\cite{HinichRothman1998}.

\begin{thm}[Reversibility gives a real factorial bispectrum]\label{thm:real-bispectrum}
If $N$ is reversible and $c_3\in L^1(\R^2)$ exists, then
\[
c_3(\tau_1,\tau_2)=c_3(-\tau_1,-\tau_2)\quad\text{a.e.},
\]
and therefore $B_{\mathrm{fac}}(\omega_1,\omega_2)\in\R$ for all $(\omega_1,\omega_2)\in\R^2$.
\end{thm}

\begin{rem}
The same argument applies at the level of finite signed measures: if the reduced third factorial cumulant is finite, reversibility gives central symmetry of that measure and its Fourier--Stieltjes transform is real. The $L^1$ assumption is imposed here only to speak of the bispectral density.
\end{rem}

\begin{lem}[Frequency-domain relation between complete and factorial third cumulants]\label{lem:ordinary-factorial-bispectrum}
Assume that the reduced second factorial cumulant measure and the reduced third factorial and complete cumulant measures are finite signed measures, so that the Fourier--Stieltjes transforms below are ordinary bounded functions. Let $B_{\mathrm{comp}}$ denote the Fourier--Stieltjes transform of $\gamma_3^{\mathrm{red}}$, the reduced complete, or ordinary, third cumulant measure, and let $B_{\mathrm{fac}}$ denote the transform of $\kappa_3^{\mathrm{red}}$, the reduced third factorial cumulant measure, under the same Fourier convention. When the latter measure has an $L^1$ density, $B_{\mathrm{fac}}$ is its factorial bispectrum. Then
\begin{equation}\label{eq:complete-factorial-bispectrum-relation}
B_{\mathrm{comp}}(\omega_1,\omega_2)
=
B_{\mathrm{fac}}(\omega_1,\omega_2)
+\Gamma(\omega_1)+\Gamma(\omega_2)+\Gamma(\omega_1+\omega_2)-2\lambda.
\end{equation}
In particular, since the second reduced factorial cumulant is symmetric and $\Gamma$ is real-valued,
\[
\Imm B_{\mathrm{comp}}=\Imm B_{\mathrm{fac}}.
\]
The finite-measure formulation above is the one used below; distributional variants require a separate specification of test-function classes and diagonal embeddings.
\end{lem}

The three $\Gamma$ terms in \eqref{eq:complete-factorial-bispectrum-relation} are the three possible pair diagonals, while the final $-2\lambda$ is the net effect of replacing the three $\widehat\kappa_2^{\mathrm{red}}$ contributions and the triple diagonal by full Bartlett spectra.

\subsection{A frequency-domain formula for branching-cluster third cumulants}

Let $0<m<1$, and let $h$ be a probability density on $\R$. Consider the stationary Poisson branching-cluster process on $\R$ with immigrant rate $\nu$ and offspring displacement intensity measure $m h(x)\,dx$. Put
\[
\Phi(\omega):=m\widehat h(\omega),\qquad R(\omega):=\frac{1}{1-\Phi(\omega)},\qquad \lambda:=\frac{\nu}{1-m}.
\]

Related higher-order cumulant formulae for Hawkes processes appear in \cite{JovanovicHertzRotter2015} and \cite{Privault2021}. The following derivation is included to fix the normalization and to emphasize that the formula itself does not require one-sided support.

\begin{thm}[Frequency-domain formula for the reduced complete third cumulant]\label{thm:branching-bispectrum}
For the stationary Poisson branching-cluster process on $\R$ described above, the Fourier--Stieltjes transform of the reduced complete third cumulant measure is
\begin{equation}\label{eq:branching-complete-bispectrum-R-form}
B_{\mathrm{comp}}(\omega_1,\omega_2)
=
\lambda R(\omega_1)R(\omega_2)R(\omega_3)
\{R(-\omega_1)+R(-\omega_2)+R(-\omega_3)-2\},
\qquad \omega_3:=-\omega_1-\omega_2.
\end{equation}
Equivalently,
\begin{equation}\label{eq:branching-complete-bispectrum-Q-form}
B_{\mathrm{comp}}(\omega_1,\omega_2)
=\lambda\,
\frac{1-m^2 Q(\omega_1,\omega_2)}
{|1-m\widehat h(\omega_1)|^2 |1-m\widehat h(\omega_2)|^2 |1-m\widehat h(-\omega_1-\omega_2)|^2},
\end{equation}
where
\[
Q(\omega_1,\omega_2)
:=\widehat h(-\omega_1)\widehat h(-\omega_2)
+\widehat h(\omega_1+\omega_2)
\{\widehat h(-\omega_1)+\widehat h(-\omega_2)-2m\widehat h(-\omega_1)\widehat h(-\omega_2)\}.
\]
The denominator in \eqref{eq:branching-complete-bispectrum-Q-form} is real and strictly positive. Hence
\[
\Imm B_{\mathrm{comp}}(\omega_1,\omega_2)
=
-\lambda m^2\,
\frac{\Imm Q(\omega_1,\omega_2)}
{|1-m\widehat h(\omega_1)|^2 |1-m\widehat h(\omega_2)|^2 |1-m\widehat h(-\omega_1-\omega_2)|^2}.
\]
\end{thm}

For a one-sided temporal Hawkes process, $h$ is supported on $(0,\infty)$. Thus the algebra used later for one-sided kernels is not intrinsically causal; causality enters only through the support and phase of $\widehat h$.

\begin{cor}[Centrally symmetric branching kernels have real third-order spectra]\label{cor:symmetric-branching-real}
In the setting of Theorem~\ref{thm:branching-bispectrum}, let $C$ denote the rooted branching cluster. Suppose $h(x)=h(-x)$ a.e. Then the branching-cluster law satisfies $C\stackrel d=-C$, the stationary process is reversible, and both the complete and factorial third-order Fourier--Stieltjes transforms are real-valued. The relevant finite moments hold because the subcritical Poisson Galton--Watson cluster has finite third moment. Indeed, when the offspring distribution is Poisson with mean \(m<1\), the total progeny has the Borel distribution, which has finite exponential moments in a neighbourhood of the origin; see Dwass~\cite{dwass1969total}. If the reduced third factorial cumulant has an $L^1(\R^2)$ density, then the factorial bispectral density is real. The complete third cumulant still contains the usual diagonal singular components; by Lemma~\ref{lem:ordinary-factorial-bispectrum}, their Fourier--Stieltjes contribution is real.
\end{cor}

\begin{rem}[Symmetry is about the origin]
The symmetry in Corollary~\ref{cor:symmetric-branching-real} is central symmetry under $x\mapsto -x$. A delay density that is symmetric around a nonzero midpoint, for example the one-sided uniform density $a^{-1}\1_{(0,a)}$, is not symmetric in this sense. This distinction matters later when we discuss exceptional one-sided kernels.
\end{rem}

\begin{prop}[Continuous-displacement branching clusters lie in the $L^1$ regime]\label{prop:L1-branching}
Let $C$ be a rooted branching cluster with a.s. finite genealogy, $\E[C(\R)^3]<\infty$, and conditionally independent edge displacements admitting Lebesgue densities on $\R$. Then the reduced third factorial cumulant measure of the stationary Poisson cluster process generated by $C$ is absolutely continuous with some density $c_3\in L^1(\R^2)$. The same conclusion applies, in particular, to one-sided branching clusters whose edge displacements have densities on $(0,\infty)$.
\end{prop}

\subsection{Sign symmetrization and matched second order}

The final construction in this section is the generic orientation-erasing device. Let $C$ be a random finite simple counting measure on $\R$ such that
\[
C(\{0\})=1\ \text{a.s.},\qquad
\mathrm{supp}(C)\subset[0,\infty)\ \text{a.s.},\qquad
\E[C(\R)^3]<\infty.
\]
Let $\{U_j\}$ be a homogeneous Poisson process on $\R$ with rate $\nu>0$, independent of i.i.d. copies $\{C_j\}$ of $C$. For $s\in\{\pm1\}$ and a counting measure $\mu$, write $s\mu$ for the pushforward of $\mu$ under $x\mapsto sx$, that is,
\[
(s\mu)(A):=\mu(sA),\qquad A\subset\R\ \text{Borel}.
\]
For $\theta\in[-1,1]$, let $\{S_{j,\theta}\}_j$ be i.i.d. signs with
\[
\PP(S_{j,\theta}=+1)=\frac{1+\theta}{2},\qquad
\PP(S_{j,\theta}=-1)=\frac{1-\theta}{2},
\]
independent of $\{U_j,C_j\}_j$, and set
\[
C_{j,\theta}:=S_{j,\theta}C_j,
\qquad
N_\theta:=\sum_j \vartheta_{U_j}C_{j,\theta}.
\]
Thus $N_1$ is the forward one-sided cluster process, $N_{-1}$ is its reversed orientation, and $N_0$ is the sign-symmetrized reversible null. This is a per-cluster sign randomization: each immigrant cluster is independently flipped or not flipped. Equivalently, the cluster law $\Law(C_{1,\theta})$ is a convex combination of $\Law(C)$ and $\Law(-C)$, but $N_\theta$ is not, in general, a process-level convex mixture of $N_1$ and $RN_1$.

\begin{prop}[Matched sign-symmetrized family]\label{prop:sign-symmetrized-family}
For every $\theta\in[-1,1]$, $N_\theta$ has intensity
\[
\lambda=\nu\,\E[C(\R)],
\]
and the same reduced second factorial cumulant measure, hence the same full Bartlett spectrum
\[
\Gamma(\omega)=\nu\,\E |\widehat C(\omega)|^2,
\qquad
\widehat C(\omega):=\int_\R e^{-\ii\omega x}\,C(dx).
\]
Moreover, $N_{-1}\stackrel d=RN_1$ and $N_0$ is reversible. If the reduced third factorial cumulant density of $N_1$ exists and belongs to $L^1(\R^2)$, then
\[
c_{3,-1}(\tau)=c_{3,1}(-\tau)\quad\text{a.e.},
\]
and
\[
c_{3,\theta}=c_3^{\mathrm e}+\theta c_3^{\mathrm o},
\qquad
B_\theta=\Ree B_1+\ii\theta\,\Imm B_1,
\]
where
\[
c_3^{\mathrm e}(\tau):=\tfrac12(c_{3,1}(\tau)+c_{3,1}(-\tau)),
\qquad
c_3^{\mathrm o}(\tau):=\tfrac12(c_{3,1}(\tau)-c_{3,1}(-\tau)).
\]
\end{prop}

\section{A reversible spectral match for monotone kernels}

Section~2 already gave a generic orientation-erased comparison by independently reflecting whole clusters. The aim here is different and stronger: within the branching setting itself, we construct an explicit reversible branching-cluster model that matches the full Bartlett spectrum of a broad class of one-sided kernels.

Let $0<m<1$ and $h:\mathbb{R}\rightarrow[0,\infty)$ be a probability density satisfying $h(x)=0$ when $x<0$. Write $\check h(x):=h(-x)$. Since \(h(x)=0\) for \(x<0\), the Fourier convention above gives
\[
\widehat h(\omega)=\int_0^\infty e^{-\ii\omega t}h(t)\,dt.
\]
The stationary one-sided branching model with immigrant rate $\nu$ and offspring density $mh$ has total intensity $\lambda=\nu/(1-m)$ and full Bartlett spectrum
\[
f_h(\omega)=\frac{\lambda}{|1-m\widehat h(\omega)|^2}.
\]
This is the preceding Bartlett spectrum $\Gamma$ specialized to the one-sided branching model; see \cite{daley2003introduction,Hawkes1971} for Hawkes spectra and point-process spectral conventions.

\begin{thm}[Reversible spectral match for monotone kernels]\label{thm:monotone-match}
Assume that $h$ has a nonincreasing representative on $(0,\infty)$. Define $\rho_h$ for a.e. $x\in\mathbb{R}$ by
\begin{equation}\label{eq:rho-h-def}
\rho_h(x):=\frac{h(|x|)-m(h*\check h)(x)}{2-m}.
\end{equation}
Then $\rho_h$ admits an even probability-density representative on $\mathbb{R}$; in particular, $\widehat{\rho_h}$ is real-valued. Let
\begin{equation}\label{eq:pn-def}
p_n:=\frac{(2n-2)!}{2^{2n-1}n!(n-1)!}\,\frac{[m(2-m)]^n}{m},\qquad n\ge1.
\end{equation}
Then $\sum_{n\ge1}p_n=1$. If $K$ has law $\PP(K=n)=p_n$, if $Y_1,Y_2,\ldots$ are i.i.d. with any even density representative of $\rho_h$, if $K$ is independent of $Y_1,Y_2,\ldots$, and if
\[
Y=\sum_{k=1}^K Y_k,
\]
then $Y$ has an even density $\varphi_h$ and
\begin{equation}\label{eq:varphi-transform}
\widehat{\varphi_h}(\omega)=\frac{1-\sqrt{1-m(2-m)\widehat{\rho_h}(\omega)}}{m},
\end{equation}
where the square root is the positive real square root. Indeed, since $\rho_h$ is even, $\widehat{\rho_h}$ is real-valued, and
\[
1-m(2-m)\widehat{\rho_h}(\omega)
=
|1-m\widehat h(\omega)|^2
\ge (1-m)^2>0,
\]
so the branch is unambiguous. Moreover,
\begin{equation}\label{eq:spectral-match}
|1-m\widehat{\varphi_h}(\omega)|^2=|1-m\widehat h(\omega)|^2.
\end{equation}
Consequently, because $\varphi_h$ is even, the stationary branching-cluster process with the same immigrant rate $\nu$ and offspring density $m\varphi_h$ is reversible, and hence serves as an undirected comparison null in our sense. It has the same total intensity and full Bartlett spectrum as the one-sided model with offspring density $mh$. In the subcritical branching setting the complete reduced covariance measures are finite, so uniqueness of Fourier--Stieltjes transforms implies that these complete reduced covariance measures also agree. In addition, Corollary~\ref{cor:symmetric-branching-real} shows that the reversible match has real complete and factorial third-order Fourier--Stieltjes transforms. If the reduced third factorial cumulant has an $L^1(\R^2)$ density, then its factorial bispectral density is real; the complete third cumulant differs only by the usual real diagonal Fourier--Stieltjes terms from Lemma~\ref{lem:ordinary-factorial-bispectrum}.
\end{thm}

Two immediate corollaries are worth recording. For the exponential kernel $h(t)=\beta e^{-\beta t}\1_{\{t>0\}}$, one gets
\[
\rho_h(x)=\frac{\beta}{2}e^{-\beta|x|},
\]
so the theorem recovers the Laplace building block from the exponential case. For the Lomax family
\[
h_\alpha(t)=\alpha(1+t)^{-1-\alpha}\1_{\{t>0\}},\qquad \alpha>0,
\]
the same construction yields an explicit reversible spectral match with
\[
\frac{1-m}{2-m}h_\alpha(|x|)\le \rho_{h_\alpha}(x)\le \frac{1}{2-m}h_\alpha(|x|),
\]
so the symmetric building block $\rho_{h_\alpha}$ has the same polynomial tail order as the original one-sided kernel. This is the heavy-tailed input used in the random-sum construction of $\varphi_h$; we do not need a separate tail asymptotic for the final offspring density here.

\begin{cor}\label{cor:L1-undirected-match}
The reversible branching-cluster process in Theorem~\ref{thm:monotone-match} has a reduced third factorial cumulant density in $L^1(\R^2)$.
\end{cor}

\begin{proof}
The associated Galton--Watson genealogy is subcritical with Poisson
offspring mean \(m<1\). Its total progeny has the Borel law and hence
finite third moment; see Dwass~\cite{dwass1969total}. The edge
displacements have the Lebesgue density \(\varphi_h\) on \(\R\).
Proposition~\ref{prop:L1-branching} therefore applies.
\end{proof}

\section{Observable contrasts and non-vacuity}

This section has two complementary aims. We first turn the population odd third-order component into observable finite-window orientation contrasts. We then show that the corresponding population signal is genuinely present for broad one-sided kernels.

\subsection{Finite-window orientation contrasts}

The finite-window diagnostics corresponding to the preceding results are orientation contrasts. For an anchor event $x_0$, the ordered triple-lag pattern $(\tau_1,\tau_2)$ is compared with its reversed orientation $(-\tau_1,-\tau_2)$.

Let $g:\R^2\to\R$ be bounded Borel, compactly supported, and jointly odd, i.e.
\[
g(-\tau_1,-\tau_2)=-g(\tau_1,\tau_2).
\]
Define
\[
\mathcal O_{T,g}
:=
\frac{1}{T}
\sum_{x_0,x_1,x_2\in N\cap[0,T]}^{\neq}
g(x_1-x_0,x_2-x_0).
\]

For bounded compactly supported $f:\R^2\to\R$, the order-three factorial moment--cumulant expansion under the reduced-measure convention above is
\begin{align}
\iint f\,d\alpha_{3,\theta}^{\mathrm{red}}
&=
\lambda^3\iint f(\tau_1,\tau_2)\,d\tau_1d\tau_2 
+\lambda\iint f(\tau_1,\tau_2)\,\kappa_2^{\mathrm{red}}(d\tau_1)\,d\tau_2 \notag\\
&\quad+\lambda\iint f(\tau_1,\tau_2)\,d\tau_1\,\kappa_2^{\mathrm{red}}(d\tau_2)
+\lambda\iint f(v,v+u)\,dv\,\kappa_2^{\mathrm{red}}(du) \notag\\
&\quad+\iint f\,d\kappa_{3,\theta}^{\mathrm{red}}.
\label{eq:alpha3-red-expansion}
\end{align}
To derive \eqref{eq:alpha3-red-expansion}, start from the order-three factorial moment--cumulant partition identity in full coordinates, as in Brillinger's cumulant formalism for stationary point processes \cite{Brillinger1975}:
\[
\alpha_3
=
\kappa_1^{\otimes3}
+\kappa_2^{01}\otimes\kappa_1^2
+\kappa_2^{02}\otimes\kappa_1^1
+\kappa_2^{12}\otimes\kappa_1^0
+\kappa_3,
\]
where the superscripts indicate the coordinate blocks, not powers; for example, $\kappa_2^{01}\otimes\kappa_1^2$ acts with $\kappa_2$ on $(t_0,t_1)$ and $\kappa_1$ on $t_2$.
Testing this full-coordinate identity against
\[
F(t_0,t_1,t_2)=\phi(t_0)f(t_1-t_0,t_2-t_0),
\qquad \phi\in C_c(\R),\quad \int \phi=1,
\]
and then using stationarity gives the reduced-coordinate expansion.
The two pair blocks involving the anchor \(t_0\) give
\[
\lambda\iint f(\tau_1,\tau_2)\,
\kappa_2^{\mathrm{red}}(d\tau_1)\,d\tau_2
\quad\text{and}\quad
\lambda\iint f(\tau_1,\tau_2)\,
d\tau_1\,\kappa_2^{\mathrm{red}}(d\tau_2).
\]
The remaining pair block, corresponding to \((t_1,t_2)\), gives, with
\(u=t_2-t_1\) and \(v=t_1-t_0\),
\[
\lambda\iint f(v,v+u)\,dv\,\kappa_2^{\mathrm{red}}(du).
\]
The product block and the connected block give respectively the first
and last terms in \eqref{eq:alpha3-red-expansion}.

The first four terms depend only on the common intensity and common second-order structure from Proposition~\ref{prop:sign-symmetrized-family}, while the connected term $\kappa_{3,\theta}^{\mathrm{red}}$ carries the sign dependence. For a fixed anchor these disconnected contributions need not vanish termwise against an odd $g$; rather, their common contribution is the intercept at $\theta=0$, and the finite-window reflection argument below shows that this intercept is zero.

\begin{prop}[Odd local contrast for the sign-symmetrized family]\label{prop:odd-kernel}
In the setting of Proposition~\ref{prop:sign-symmetrized-family}, assume that $c_{3,1}\in L^1(\R^2)$ and let $c_3^{\mathrm o}$ be the odd part defined there. Then
\[
\E_\theta[\mathcal O_{T,g}]=\theta\,\mu_{T,g},
\]
where
\[
\mu_{T,g}
=
\frac{1}{T}\int_0^T\iint_{\R^2}
g(\tau_1,\tau_2)\,
\1_{\{0\le x_0+\tau_1\le T,\ 0\le x_0+\tau_2\le T\}}
\,c_3^{\mathrm o}(\tau_1,\tau_2)\,d\tau_1d\tau_2\,dx_0.
\]
If $\mathrm{supp}(g)\subset[-H,H]^2$, then, as $T\to\infty$,
\[
\mu_{T,g}=\mu_g+O(T^{-1}),
\qquad
\mu_g:=\iint_{\R^2} g(\tau_1,\tau_2)c_3^{\mathrm o}(\tau_1,\tau_2)\,d\tau_1d\tau_2.
\]
If in addition $\widehat g=\ii H_g$ with $H_g$ real-valued and integrable, then
\[
\mu_g
=
\frac{1}{(2\pi)^2}
\iint_{\R^2}H_g(\omega_1,\omega_2)\,\Imm B_1(\omega_1,\omega_2)\,d\omega_1d\omega_2.
\]
\end{prop}

Equation~\eqref{eq:alpha3-red-expansion} shows that matched intensity and matched second order remove all $\theta$-dependence from the disconnected terms; only the connected contribution $\kappa_{3,\theta}^{\mathrm{red}}$ varies with the sign parameter. At $\theta=0$, the window reflection on $[0,T]$ is $r_T=\vartheta_T\circ R$, so stationarity and reversibility make the statistic antisymmetric. The appendix gives the exact finite-window calculation.

\begin{prop}[Optimal bounded local odd contrast]\label{prop:optimal-local-contrast}
Assume $c_3^{\mathrm o}\in L^1(\R^2)$ and, for $H>0$, define
\[
D_H
:=
\sup\Big\{
\mu_g:\ g\text{ Borel},\ |g|_\infty\le1,\ \mathrm{supp}(g)\subset[-H,H]^2,\ g(-\tau)=-g(\tau)
\Big\},
\]
where
\[
\mu_g=\iint_{\R^2} g(\tau_1,\tau_2)c_3^{\mathrm o}(\tau_1,\tau_2)\,d\tau_1d\tau_2.
\]
Then
\[
D_H=\int_{[-H,H]^2}|c_3^{\mathrm o}(\tau)|\,d\tau,
\qquad
D_H\uparrow \|c_3^{\mathrm o}\|_1
\quad\text{as }H\to\infty.
\]
In particular, $c_3^{\mathrm o}\not\equiv0$ if and only if there exists a
bounded Borel, compactly supported, jointly odd $g$ with $\mu_g\neq0$. If one restricts instead to continuous or smooth compactly supported odd test functions, the same conclusion follows by standard $L^1$ approximation of $\operatorname{sgn}(c_3^{\mathrm o})$ on compact sets.
\end{prop}

\subsection{Non-vacuity of the third-order signal}

The next elementary representation will be used in the finite-second-moment monotone case.

\begin{lem}[Monotone densities as scale mixtures of uniforms]\label{lem:uniform-mixture}
Let $h$ be a nonincreasing probability density on $(0,\infty)$, and let $X$ have density $h$. Then there exists a nonnegative random variable $Z$ and an independent $Y\sim\mathrm{Unif}(0,1)$ such that
\[
X\stackrel d=YZ.
\]
Equivalently, $h$ admits the representation
\[
h(x)=\int_{(x,\infty)} z^{-1}\,G(dz)
\quad\text{for a.e. }x>0,
\]
where $G=\Law(Z)$. Moreover, for every $p>0$,
\[
\E X^p=\frac{1}{p+1}\E Z^p
\]
whenever either side is finite. Finally,
\[
h(x)=a^{-1}\1_{(0,a)}(x)\quad\text{a.e.}
\]
if and only if $Z=a$ a.s.
\end{lem}

\begin{proof}
The scale-mixture representation is the $k=1$ case of Williamson's representation of multiply monotone functions, equivalently the classical representation of nonincreasing densities as mixtures of uniform densities; see \cite{Williamson1956} and the continuous/discrete treatment in \cite{Lefevre_and_Loisel2013}. Conditional on $Z$, the density of $YZ$ is $z^{-1}\1_{(0,z)}(x)$, so the displayed integral representation follows. The moment identity is immediate from
\[
\E[(YZ)^p\mid Z]=Z^p\E[Y^p]=\frac{Z^p}{p+1}.
\]
If $Z=a$ a.s., then $YZ$ has density $a^{-1}\1_{(0,a)}$. Conversely, if $YZ$ has this density, then $YZ\le a$ a.s.; if $\PP(Z>a+\varepsilon)>0$ for some $\varepsilon>0$, then $\PP(YZ>a)>0$, a contradiction. Thus $Z\le a$ a.s. Since $\E[YZ]=a/2$, the moment identity gives $\E Z=a$, and $Z\le a$ then forces $Z=a$ a.s.
\end{proof}

The next result shows that the imaginary-bispectrum diagnostic is not merely formal. For broad one-sided Hawkes kernels, the forward law has a genuinely nonzero imaginary bispectrum at some diagonal frequency. The complete branching-cluster formula used in the proof is Theorem~\ref{thm:branching-bispectrum}; explicit cumulant-density and family-tree formulae for Hawkes processes are given in \cite{JovanovicHertzRotter2015}, and recursive computation through the probability generating functional is developed in \cite{Privault2021}.

\begin{prop}[Non-vacuity of the imaginary bispectrum]\label{prop:nonzero-imag-bispectrum}
Let $0<m<1$, let $h$ be a probability density supported on $(0,\infty)$, and consider the stationary one-sided linear Hawkes, equivalently Poisson branching-cluster, process with immigrant rate $\nu$ and offspring density $mh$. Let $B$ denote its reduced third factorial-cumulant bispectrum. Suppose either
\begin{enumerate}
\item
\[
\int_x^\infty h(u)\,du\sim x^{-\alpha}L(x),\qquad x\to\infty,
\]
for some $0<\alpha\le2$ and slowly varying $L$; or
\item $h$ is nonincreasing on $(0,\infty)$, $\int_0^\infty u^2h(u)\,du<\infty$, and $h$ is not of the form $a^{-1}\1_{(0,a)}$ for some $a>0$.
\end{enumerate}
Then there exists $t_0>0$ such that
\[
\Imm B(t_0,t_0)\neq0.
\]
Consequently, this one-sided process is separated at third order from every reversible process whose reduced third factorial cumulant has an $L^1$ density.
\end{prop}

The proof is in the appendix. The one-sided uniform density is the exceptional one-sided cluster case in which the imaginary bispectrum vanishes identically despite the one-sided construction; the appendix records the resulting phase-cancellation algebra, shows that condition~(ii) is sharp within the monotone finite-second-moment class, and collects the quantitative small-frequency lower bounds that follow from the proof.

\begin{rem}[A crude branching-ratio envelope]\label{rem:branching-ratio-envelope}
Let $B_{\mathrm{fac}}$ be the reduced third factorial-cumulant
bispectrum of a forward Poisson branching-cluster process, and let
\(M:=C(\R)\) be the total size of one cluster. If the reduced third
factorial cumulant has an \(L^1\) density, then
\[
\sup_{(\omega_1,\omega_2)\in\R^2}
|\Imm B_{\mathrm{fac}}(\omega_1,\omega_2)|
\le
\nu\,\E[(M)_3].
\]
For Poisson offspring with mean $m$,
\[
\E[(M)_3]
=
\frac{m^2(2m^2-8m+9)}{(1-m)^5}.
\]
Thus, at fixed total intensity $\lambda=\nu/(1-m)$, the upper envelope is
\[
\lambda\,\frac{m^2(2m^2-8m+9)}{(1-m)^4}
=O(m^2),
\qquad m\downarrow0.
\]
This is only a crude global envelope; the actual signal at a given frequency also depends on the kernel shape and time scale. Under the scale family $h_\beta(t)=\beta h(\beta t)$, the corresponding third-cumulant density scales as
\[
c_3^{(\beta)}(\tau_1,\tau_2)=\beta^2c_3^{(1)}(\beta\tau_1,\beta\tau_2),
\]
and hence
\[
B^{(\beta)}(\omega_1,\omega_2)=B^{(1)}(\omega_1/\beta,\omega_2/\beta).
\]
The proof is included in the appendix.
\end{rem}

\section{Discussion}

The results show that a good second-order Hawkes fit should not automatically be read as evidence for a Hawkes mechanism. Once the comparison class is allowed to include the matched reversible cluster laws constructed here, intensity, pair correlation, covariance, full Bartlett spectra, and related second-order diagnostics can support clustering without supporting directed temporal self-excitation.

The relevant evidence for orientation must come from elsewhere. For prediction, a Hawkes model may still be natural because its conditional intensity is adapted to the observed history, and predictive performance can be assessed through likelihood, residual calibration, or short-horizon count prediction. For a causal or mechanistic interpretation, however, the data must contain time-asymmetric information. In the framework studied here, a robust certificate for that information appears in third order: in the odd part of the reduced third factorial cumulant, or equivalently in the imaginary part of the bispectrum.

The branching-cluster formula clarifies the null side. The Fourier--Stieltjes formula for the reduced complete third cumulant measure does not depend on one-sidedness; it applies equally to branching clusters whose displacement law is not supported on a half-line. If the offspring displacement density is centrally symmetric about the origin, then the cluster law is invariant under reflection and the null is reversible, so the factorial Fourier--Stieltjes transform is real, and any factorial bispectral density is real when it exists. Thus the distinction is not between ``Hawkes formula'' and ``spatial formula'', but between oriented and unoriented cluster laws.

The monotone-kernel construction shows that the second-order ambiguity is broad rather than exceptional. Every nonincreasing one-sided branching kernel considered in Theorem~\ref{thm:monotone-match} has an explicit reversible branching-cluster spectral match, including exponential and monotone heavy-tailed kernels. Thus familiar second-order covariance shapes should not by themselves be taken as evidence for causal self-excitation. Proposition~\ref{prop:nonzero-imag-bispectrum} supplies the complementary positive statement: for regularly varying one-sided kernels and for nonuniform monotone finite-second-moment kernels, the imaginary bispectrum is nonzero somewhere. In these cases the third-order orientation signal is not vacuous.

Several limitations remain. Reversibility implies a real bispectrum in the $L^1$ third-cumulant setting, but a real bispectrum alone does not prove reversibility. Remark~\ref{rem:uniform-kernel-exception} shows the precise limitation: for the one-sided uniform kernel, the third-order phase cancellation makes the imaginary bispectrum vanish identically even though the branching construction is one-sided. That calculation proves a phase cancellation in the third-order frequency-domain cumulant; it does not by itself establish non-reversibility of the corresponding unlabeled stationary point process. Thus a nonzero imaginary bispectrum should be read as a sufficient certificate against reversible $L^1$ nulls, not as a necessary feature of every one-sided construction. The observable statistics above give natural third-order orientation contrasts, but a full asymptotic testing theory, especially under heavy-tailed lifetimes and long memory, remains open. The explicit orientation-contrast construction and monotone spectral match developed here are univariate; in multivariate settings, directionality may enter through marks, components, or cross-excitation structure. The main message is therefore not that Hawkes models should be avoided, but that evidence for directed self-excitation must come from predictive performance, modelling context, external information, or higher-order time-asymmetric structure.


\acks During the preparation of this work the authors used ChatGPT 5.5 Pro, an AI language model, for language editing and mathematical-check assistance. The authors take full responsibility for all content.

\fund This work was supported by the Royal Society of New Zealand Marsden Fund under grants MFP-UOO2323 and MFP-UOO2518.

\competing There were no competing interests to declare which arose during the preparation or publication process of this article.


\appendix
\section{Proofs and technical remarks}

\subsection{Proof of Lemma~\ref{lem:pcp-cumulant-functional}.}
On any domain where the displayed logarithms are finite, the two generating-functional identities follow from the exponential formula for a Poisson process of immigrants. Indeed, each immigrant at $u$ contributes an independent cluster functional, and the logarithm of the Poisson product functional is the immigrant intensity integrated against the mean cluster contribution minus one.

For the finite-order identities under the weaker assumption $\E[C(\R)^n]<\infty$, apply the preceding formula first to bounded-size truncations of the cluster and then let the truncation level tend to infinity. For fixed compactly supported directions $f_0,\ldots,f_{n-1}$, the relevant $n$th derivatives are bounded by a constant times $C(\R)^n$; after integration over the immigrant location $u$, the compact-overlap indicators contribute only a finite deterministic factor. Dominated convergence therefore justifies the passage to the limit. Differentiating $K_{\mathrm{fac}}^N$ in the directions $f_0,\ldots,f_{n-1}$ gives the factorial cumulants and selects ordered tuples of pairwise distinct cluster points. Differentiating $K_{\mathrm{comp}}^N$ gives the complete cumulants and allows repeated cluster points:
\[
\left.\partial_{s_0}\cdots\partial_{s_{n-1}}
\exp\left\{\sum_{j=0}^{n-1}s_jC(f_j(u+\cdot))\right\}\right|_{s=0}
=
\prod_{j=0}^{n-1}C(f_j(u+\cdot)).
\]
Rewriting these full-coordinate formulae in stationary reduced coordinates
\[
(\tau_1,\ldots,\tau_{n-1})=(x_1-x_0,\ldots,x_{n-1}-x_0)
\]
yields the two reduced-measure identities. The Fourier--Stieltjes identity for $\gamma_3^{\mathrm{red}}$ follows by transforming the reduced complete-cumulant formula with $n=3$.

\subsection{Proof of Theorem~\ref{thm:real-bispectrum}.}
Factorial cumulants are determined by the law of the point process. Hence $N\stackrel d=RN$ implies that the full third factorial cumulant measure $\kappa_3$ is invariant under
\[
(t_0,t_1,t_2)\mapsto(-t_0,-t_1,-t_2).
\]
To pass this statement to reduced coordinates, let $\psi$ be a bounded compactly supported test function on $\R^2$ and let $\phi$ be an even compactly supported test function on $\R$ with $\int\phi=1$. Testing the full measure against
\[
F(t_0,t_1,t_2)=\phi(t_0)\psi(t_1-t_0,t_2-t_0)
\]
and using the definition of the reduced measure gives $\int\psi\,d\kappa_3^{\mathrm{red}}$. Reflection invariance gives the same integral with $\psi(\tau_1,\tau_2)$ replaced by $\psi(-\tau_1,-\tau_2)$. Hence
\[
\kappa_3^{\mathrm{red}}(A)=\kappa_3^{\mathrm{red}}(-A),
\qquad A\subset\R^2\ \text{Borel}.
\]
If $\kappa_3^{\mathrm{red}}(d\tau)=c_3(\tau)\,d\tau$, then
\[
c_3(\tau)=c_3(-\tau)\quad\text{a.e.}
\]
Consequently
\[
B_{\mathrm{fac}}(\omega)
=
\int e^{-\ii\omega\cdot\tau}c_3(\tau)\,d\tau
=
\int e^{\ii\omega\cdot\tau}c_3(\tau)\,d\tau
=
\overline{B_{\mathrm{fac}}(\omega)},
\]
so $B_{\mathrm{fac}}$ is real-valued.

\subsection{Proof of Lemma~\ref{lem:ordinary-factorial-bispectrum}.}
Use \eqref{eq:ordinary-factorial-generating-functional} first in a setting where the generating functionals are finite, or equivalently apply the moment--cumulant partition formula and then pass from bounded truncations under the stated finite-measure assumptions. Differentiating the identity
\[
K_{\mathrm{comp}}(f)=K_{\mathrm{fac}}(e^f-1)
\]
at zero in the three directions $f_0,f_1,f_2$ gives
\begin{align*}
\gamma_3(f_0,f_1,f_2)
&=
\kappa_3(f_0,f_1,f_2)
+\kappa_2(f_0f_1,f_2)
+\kappa_2(f_0f_2,f_1) \\
&\quad+\kappa_2(f_1f_2,f_0)
+\lambda\int f_0f_1f_2 .
\end{align*}
Thus the complete third cumulant equals the third factorial cumulant plus the three pair-diagonal second-factorial-cumulant terms and the full triple diagonal. Passing to reduced Fourier coordinates with frequencies $(\omega_1,\omega_2,-\omega_1-\omega_2)$, the two pair diagonals involving the anchor coordinate contribute
\[
\widehat\kappa_2^{\mathrm{red}}(\omega_1),\qquad
\widehat\kappa_2^{\mathrm{red}}(\omega_2).
\]
For the pair diagonal $t_1=t_2$, an ordered-coordinate convention may first produce $\widehat\kappa_2^{\mathrm{red}}(-\omega_1-\omega_2)$; this equals $\widehat\kappa_2^{\mathrm{red}}(\omega_1+\omega_2)$ because the second reduced factorial cumulant is symmetric, $\kappa_2^{\mathrm{red}}(A)=\kappa_2^{\mathrm{red}}(-A)$. Hence the three pair diagonals contribute
\[
\widehat\kappa_2^{\mathrm{red}}(\omega_1),\qquad
\widehat\kappa_2^{\mathrm{red}}(\omega_2),\qquad
\widehat\kappa_2^{\mathrm{red}}(\omega_1+\omega_2),
\]
and the triple diagonal contributes $\lambda$. Since
\[
\Gamma(\omega)=\lambda+\widehat\kappa_2^{\mathrm{red}}(\omega),
\]
we obtain
\[
B_{\mathrm{comp}}
=
B_{\mathrm{fac}}
+\{\Gamma(\omega_1)-\lambda\}
+\{\Gamma(\omega_2)-\lambda\}
+\{\Gamma(\omega_1+\omega_2)-\lambda\}
+\lambda,
\]
which is \eqref{eq:complete-factorial-bispectrum-relation}. The same symmetry of $\kappa_2^{\mathrm{red}}$ makes the correction terms real-valued, so $\Imm B_{\mathrm{comp}}=\Imm B_{\mathrm{fac}}$.

\subsection{Proof of Theorem~\ref{thm:branching-bispectrum}.}
Let $C$ be the rooted branching cluster and write
\[
W(\omega):=\widehat C(\omega)=\sum_{x\in C}e^{-\ii\omega x}.
\]
By Lemma~\ref{lem:pcp-cumulant-functional}, in reduced Fourier coordinates with $\omega_1+\omega_2+\omega_3=0$,
\begin{equation}\label{eq:poisson-cluster-complete-cumulant-moment}
B_{\mathrm{comp}}(\omega_1,\omega_2)=\nu\,\E\{W(\omega_1)W(\omega_2)W(\omega_3)\}.
\end{equation}

The cluster recursion is
\[
W(\omega)=1+\sum_j e^{-\ii\omega D_j}W_j(\omega),
\]
where $\{D_j\}$ is a Poisson process on $\R$ with intensity $mh(x)\,dx$, and, conditionally on the first generation, the $W_j$'s are independent copies of $W$. Let
\[
M_1(a):=\E W(a),\qquad M_2(a,b):=\E\{W(a)W(b)\},\qquad M_3(a,b,c):=\E\{W(a)W(b)W(c)\}.
\]
The first moment satisfies
\[
M_1(a)=1+\Phi(a)M_1(a),
\qquad\text{hence}\qquad M_1(a)=R(a).
\]
For the second moment, the two frequency factors either descend through different first-generation children or through the same child. The Poisson moment formula gives
\[
M_2(a,b)=R(a)R(b)+\Phi(a+b)M_2(a,b),
\]
so
\begin{equation}\label{eq:branching-cluster-M2}
M_2(a,b)=R(a)R(b)R(a+b).
\end{equation}
For the third moment, the three factors may pass through three different children, through one common pair and one separate child, or through one common child. Hence
\begin{align*}
M_3(a,b,c)
&=R(a)R(b)R(c)
+\Phi(a+b)M_2(a,b)R(c)
+\Phi(a+c)M_2(a,c)R(b) \\
&\quad+\Phi(b+c)M_2(b,c)R(a)
+\Phi(a+b+c)M_3(a,b,c).
\end{align*}
When $a+b+c=0$, $\Phi(a+b+c)=m$. Using \eqref{eq:branching-cluster-M2} and multiplying by $\nu$ gives
\begin{align*}
\nu M_3(a,b,c)
&=\frac{\nu}{1-m}R(a)R(b)R(c)
\{1+\Phi(a+b)R(a+b)+\Phi(a+c)R(a+c)+\Phi(b+c)R(b+c)\} \\
&=\lambda R(a)R(b)R(c)\{R(a+b)+R(a+c)+R(b+c)-2\}.
\end{align*}
Taking $(a,b,c)=(\omega_1,\omega_2,-\omega_1-\omega_2)$ proves \eqref{eq:branching-complete-bispectrum-R-form}.

It remains only to rewrite \eqref{eq:branching-complete-bispectrum-R-form}. Put $r(\omega):=1-m\widehat h(\omega)$. Since $h$ is real-valued, $r(-\omega)=\overline{r(\omega)}$. With $a=\omega_1$, $b=\omega_2$, and $c=-a-b$, formula \eqref{eq:branching-complete-bispectrum-R-form} equals
\[
\lambda\frac{1}{r(a)r(b)r(c)}
\left\{\frac{1}{r(-a)}+\frac{1}{r(-b)}+\frac{1}{r(a+b)}-2\right\}.
\]
Multiplying numerator and denominator by $r(-a)r(-b)r(a+b)$ gives the denominator in \eqref{eq:branching-complete-bispectrum-Q-form}. The numerator is
\[
r(-b)r(a+b)+r(-a)r(a+b)+r(-a)r(-b)-2r(-a)r(-b)r(a+b).
\]
Using $r(\omega)=1-m\widehat h(\omega)$, this equals
\[
1-m^2\Big[\widehat h(-a)\widehat h(-b)
+\widehat h(a+b)\{\widehat h(-a)+\widehat h(-b)-2m\widehat h(-a)\widehat h(-b)\}\Big],
\]
which is $1-m^2Q(a,b)$. This proves \eqref{eq:branching-complete-bispectrum-Q-form}. The displayed imaginary-part identity follows because the denominator is real and bounded below by $(1-m)^6$.

\subsection{Proof of Corollary~\ref{cor:symmetric-branching-real}.}
If $h(x)=h(-x)$, then reflecting every edge displacement $D\mapsto -D$ preserves the displacement law and leaves the Galton--Watson genealogy unchanged. Hence the rooted cluster law satisfies $C\stackrel d=-C$. The subcritical Poisson Galton--Watson cluster has finite third moment, so the complete and factorial reduced third cumulant measures are finite. The stationary immigrant Poisson process is also invariant under reflection, so the full branching-cluster process is reversible. Theorem~\ref{thm:real-bispectrum} gives real $B_{\mathrm{fac}}$ in the $L^1$ factorial-density case, and the finite-measure version gives the same Fourier--Stieltjes conclusion without writing a density. Lemma~\ref{lem:ordinary-factorial-bispectrum} then gives real $B_{\mathrm{comp}}$, since the added diagonal Fourier--Stieltjes terms are real. Equivalently, $\widehat h(\omega)$ and $R(\omega)$ are real and even, so real-valuedness of the complete transform is also immediate from \eqref{eq:branching-complete-bispectrum-R-form}.

\subsection{Proof of Proposition~\ref{prop:L1-branching}.}
Label the vertices of the genealogy by finite words in a countable alphabet, with the root denoted by $\varnothing$. The random genealogy is a random finite rooted labelled tree $\mathcal T$. Conditional on $\mathcal T$, write $T_v$ for the position of vertex $v$.

By Lemma~\ref{lem:pcp-cumulant-functional},
\[
\kappa_3^{\mathrm{red}}
=
\nu\,\E
\sum_{v_0,v_1,v_2\in\mathcal T}^{\neq}
\delta_{T_{v_1}-T_{v_0}}
\otimes
\delta_{T_{v_2}-T_{v_0}}.
\]
Fix a finite labelled genealogy and an ordered triple $\mathbf v=(v_0,v_1,v_2)$ of distinct vertices. The vector
\[
X_{\mathbf v}:=(T_{v_1}-T_{v_0},\,T_{v_2}-T_{v_0})
\]
is a linear image of the edge-displacement vector. The two signed incidence vectors corresponding to the paths from $v_0$ to $v_1$ and from $v_0$ to $v_2$ are linearly independent: in a tree, two such path vectors cannot be scalar multiples unless the two paths have the same endpoint. Thus this linear map has rank two.

Conditional on $\mathcal T$, the edge-displacement vector has a Lebesgue density. A rank-two linear image of a vector with a Lebesgue density has a Lebesgue density on $\R^2$. Therefore each $X_{\mathbf v}$ has a conditional density $q_{\mathbf v}^{\mathcal T}$. Since the set of possible finite labelled genealogies is countable, these densities may be chosen consistently on the countable state space, and
\[
c_3(\tau_1,\tau_2)
=
\nu\,\E\left[
\sum_{v_0,v_1,v_2\in\mathcal T}^{\neq}
q_{\mathbf v}^{\mathcal T}(\tau_1,\tau_2)
\right]
\]
is a density of $\kappa_3^{\mathrm{red}}$. Tonelli's theorem gives
\[
\int_{\R^2}c_3(\tau_1,\tau_2)\,d\tau_1d\tau_2
=
\nu\,\E\sum_{v_0,v_1,v_2\in\mathcal T}^{\neq}1
=
\nu\,\E[(C(\R))_3]<\infty.
\]
Thus $c_3\in L^1(\R^2)$.

\subsection{Proof of Proposition~\ref{prop:sign-symmetrized-family}.}
The intensity is $\lambda=\nu\,\E C(\R)$, which is unaffected by reflection of a cluster. By Lemma~\ref{lem:pcp-cumulant-functional} with $n=2$,
\[
\kappa_{2,\theta}^{\mathrm{red}}(d\tau)
=
\nu\,\E\sum_{x,y\in C_{1,\theta}}^{\neq}\delta_{y-x}(d\tau).
\]
For a fixed cluster $C$, the ordered pair-difference measure
\[
\sum_{x,y\in C}^{\neq}\delta_{y-x}
\]
is invariant under reflection: the map $C\mapsto -C$ sends the contribution $y-x$ to $x-y$, and the ordered pair $(y,x)$ is also present in the same sum. Hence $\kappa_{2,\theta}^{\mathrm{red}}$ is independent of $\theta$. Moreover,
\[
\Gamma(\omega)
=
\lambda+
\nu\,\E\sum_{x,y\in C}^{\neq}e^{-\ii\omega(y-x)}
=
\nu\,\E|\widehat C(\omega)|^2 .
\]

The identity $N_{-1}\stackrel d=RN_1$ follows because reflection sends
\[
\sum_j\vartheta_{U_j}C_j
\]
to
\[
\sum_j\vartheta_{-U_j}(-C_j),
\]
and the reflected immigrant process $\{-U_j\}$ has the same Poisson law as $\{U_j\}$. At $\theta=0$, the cluster law is reflection invariant, so $N_0$ is reversible.

For the third-order identity, Lemma~\ref{lem:pcp-cumulant-functional} with $n=3$ shows that reduced third factorial cumulants depend linearly on the cluster law. Since
\[
\Law(C_{1,\theta})=\frac{1+\theta}{2}\Law(C)+\frac{1-\theta}{2}\Law(-C),
\]
we get
\[
\kappa_{3,\theta}^{\mathrm{red}}
=
\frac{1+\theta}{2}\kappa_{3,1}^{\mathrm{red}}
+
\frac{1-\theta}{2}\kappa_{3,-1}^{\mathrm{red}}.
\]
Reflection gives
\[
c_{3,-1}(\tau)=c_{3,1}(-\tau)\quad\text{a.e.},
\]
and therefore
\[
c_{3,\theta}=c_3^{\mathrm e}+\theta c_3^{\mathrm o}.
\]
Taking Fourier transforms yields
\[
B_\theta=\Ree B_1+\ii\theta\,\Imm B_1,
\]
because the Fourier transform of $c_{3,1}(-\cdot)$ is
\[
B_1(-\omega_1,-\omega_2)=\overline{B_1(\omega_1,\omega_2)}.
\]

\subsection{Details for Theorem~\ref{thm:monotone-match}.}
We first verify that $\rho_h$ admits an even probability-density representative. Choose the nonincreasing representative of $h$ on $(0,\infty)$. Since $h$ is supported on $[0,\infty)$ and nonincreasing on $(0,\infty)$, for a.e. $x\in\R$,
\[
(h*\check h)(|x|)
=
\int_{0}^{\infty} h(|x|+u)h(u)\,du
\le h(|x|)\int_{0}^{\infty}h(u)\,du
=
h(|x|).
\]
Thus $\rho_h\ge0$ a.e. Also
\[
\int_{\R} \rho_h(x)\,dx
=
\frac{1}{2-m}\left(\int_{\R}h(|x|)\,dx-m\int_{\R}(h*\check h)(x)\,dx\right)
=
\frac{2-m}{2-m}=1.
\]
Finally, $(h*\check h)(x)=(h*\check h)(-x)$ for every $x\in\R$, so $\rho_h(x)=\rho_h(-x)$ for a.e. $x$. Hence $\rho_h$ admits an even probability-density representative.

Next,
\[
\widehat{h(|\cdot|)}(\omega)=2\Ree \widehat h(\omega),\qquad
\widehat{h*\check h}(\omega)=|\widehat h(\omega)|^2,
\]
and therefore
\[
\widehat{\rho_h}(\omega)
=
\frac{2\Ree\widehat h(\omega)-m|\widehat h(\omega)|^2}{2-m}.
\]
This gives
\[
1-m(2-m)\widehat{\rho_h}(\omega)=|1-m\widehat h(\omega)|^2.
\]
The coefficients $p_n$ are positive, and the Maclaurin series of $1-\sqrt{1-z}$ yields
\[
\sum_{n\ge1}p_n z^n=\frac{1-\sqrt{1-m(2-m)z}}{m},\qquad |z|\le1,
\]
where the square root is the principal branch on the disc and the positive real square root for real $z\in[0,1]$. In particular,
\[
\sum_{n\ge1}p_n=\frac{1-\sqrt{1-m(2-m)}}{m}=1.
\]
Because $K$ is independent of $(Y_k)_{k\ge1}$, the density of $Y=\sum_{k=1}^K Y_k$ is
\[
\varphi_h=\sum_{n\ge1}p_n\,\rho_h^{*n},
\]
so $\varphi_h$ is even. Since $|\widehat{\rho_h}(\omega)|\le1$,
\[
\widehat{\varphi_h}(\omega)=\sum_{n\ge1}p_n\widehat{\rho_h}(\omega)^n
=\frac{1-\sqrt{1-m(2-m)\widehat{\rho_h}(\omega)}}{m}.
\]
Also $\widehat{\rho_h}$ is real-valued and
\[
1-m(2-m)\widehat{\rho_h}(\omega)=|1-m\widehat h(\omega)|^2>0,
\]
so the square root in \eqref{eq:varphi-transform} is the positive real square root. Finally,
\[
|1-m\widehat{\varphi_h}(\omega)|^2
=
1-m(2-m)\widehat{\rho_h}(\omega)
=
|1-m\widehat h(\omega)|^2.
\]
Because $\varphi_h$ is even, the offspring displacements in the associated branching construction have reflection-symmetric law. Reflecting every edge displacement therefore preserves the cluster law, so the resulting stationary branching-cluster process with offspring density $m\varphi_h$ is reversible, and hence serves as an undirected comparison null in our sense. The common branching ratio $m$ gives the same total intensity $\lambda=\nu/(1-m)$, and the spectral identity above gives the same full Bartlett spectrum as the one-sided model with offspring density $mh$. Since the complete reduced covariance measures are finite in this subcritical branching setting (the cluster size has finite second moment), uniqueness of Fourier--Stieltjes transforms gives equality of the complete reduced covariance measures as well. Corollary~\ref{cor:symmetric-branching-real} gives real complete and factorial third-order Fourier--Stieltjes transforms for this reversible match. When an $L^1(\R^2)$ reduced third factorial cumulant density is present, the corresponding factorial bispectral density is real; the complete third cumulant has the usual additional diagonal singular components, whose Fourier--Stieltjes contributions are real by Lemma~\ref{lem:ordinary-factorial-bispectrum}.

For the exponential kernel $h(t)=\beta e^{-\beta t}\1_{\{t>0\}}$, one checks that
\[
(h*\check h)(x)=\frac{\beta}{2}e^{-\beta|x|},
\]
hence $\rho_h(x)=\frac{\beta}{2}e^{-\beta|x|}$ for a.e. $x$. For the Lomax family
\[
h_\alpha(t)=\alpha(1+t)^{-1-\alpha}\1_{\{t>0\}},
\]
the bounds
\[
\frac{1-m}{2-m}h_\alpha(|x|)\le \rho_{h_\alpha}(x)\le \frac{1}{2-m}h_\alpha(|x|)
\]
hold for a.e. $x\in\R$ and follow directly from $(h*\check h)(x)\le h(|x|)$ a.e.

\subsection{Proof of Proposition~\ref{prop:odd-kernel}.}
Let
\[
W_T(x_0,\tau_1,\tau_2):=\1_{\{0\le x_0+\tau_1\le T,\ 0\le x_0+\tau_2\le T\}}.
\]
The third-order Campbell formula yields
\[
\E_\theta[\mathcal O_{T,g}]
=
\frac{1}{T}\int_0^T\iint_{\R^2}
g(\tau_1,\tau_2)W_T(x_0,\tau_1,\tau_2)\,
\alpha_{3,\theta}^{\mathrm{red}}(d\tau_1,d\tau_2)\,dx_0.
\]
Applying the order-three expansion \eqref{eq:alpha3-red-expansion} with
\[
f(\tau_1,\tau_2):=\frac{1}{T}\int_0^T g(\tau_1,\tau_2)W_T(x_0,\tau_1,\tau_2)\,dx_0
\]
shows that
\[
\E_\theta[\mathcal O_{T,g}]=A_T+\theta\,\mu_{T,g},
\]
where $A_T$ is determined by the common intensity, common second-order structure, and the even third-order component, and where
\[
\mu_{T,g}
=
\frac{1}{T}\int_0^T\iint_{\R^2}
g(\tau_1,\tau_2)W_T(x_0,\tau_1,\tau_2)c_3^{\mathrm o}(\tau_1,\tau_2)\,d\tau_1d\tau_2\,dx_0.
\]

To identify the intercept, let $r_T(x)=T-x$ on $[0,T]$. On configurations supported in $[0,T]$, this reflection is
\[
r_T=\vartheta_T\circ R.
\]
Since $N_0$ is stationary and reversible,
\[
r_T(N_0|_{[0,T]})
=
(\vartheta_T RN_0)|_{[0,T]}
\stackrel d=
N_0|_{[0,T]}.
\]
Because $g$ is jointly odd,
\[
\mathcal O_{T,g}(r_T\xi)=-\mathcal O_{T,g}(\xi)
\]
for every finite configuration $\xi\subset[0,T]$. Hence $\E_0[\mathcal O_{T,g}]=0$. But $\E_0[\mathcal O_{T,g}]=A_T$, so $A_T=0$, proving the exact identity
\[
\E_\theta[\mathcal O_{T,g}]=\theta\,\mu_{T,g}.
\]

If $\mathrm{supp}(g)\subset[-H,H]^2$ and
\[
a_T(\tau_1,\tau_2):=\frac1T\int_0^T W_T(x_0,\tau_1,\tau_2)\,dx_0,
\]
then, on the support of $g$,
\[
|a_T(\tau_1,\tau_2)-1|\le \frac{2H}{T}
\]
for $T>H$. Therefore
\[
|\mu_{T,g}-\mu_g|
\le \frac{2H}{T}\|g\|_\infty\|c_3^{\mathrm o}\|_1,
\qquad
\mu_g:=\iint_{\R^2}g(\tau_1,\tau_2)c_3^{\mathrm o}(\tau_1,\tau_2)\,d\tau_1d\tau_2.
\]
Finally, because $g$ is real and jointly odd, $\widehat g=\ii H_g$, where $H_g:=-\ii\widehat g$ is real-valued and jointly odd. Fourier inversion gives
\[
\mu_g
=
\frac{1}{(2\pi)^2}\iint \widehat g(-\omega_1,-\omega_2)\,\widehat{c_3^{\mathrm o}}(\omega_1,\omega_2)\,d\omega_1d\omega_2.
\]
By Proposition~\ref{prop:sign-symmetrized-family},
\[
\widehat{c_3^{\mathrm o}}=\ii\,\Imm B_1,
\qquad
\widehat g(-\omega_1,-\omega_2)=-\ii H_g(\omega_1,\omega_2),
\]
so
\[
\mu_g
=
\frac{1}{(2\pi)^2}\iint H_g(\omega_1,\omega_2)\,\Imm B_1(\omega_1,\omega_2)\,d\omega_1d\omega_2.
\]

\begin{lem}[Small-frequency Abelian estimates]\label{lem:small-frequency-abelian-estimates}
Let $\overline H(x):=\int_x^\infty h(u)\,du\sim x^{-\alpha}L(x)$ with $0<\alpha\le2$, where $L$ is slowly varying. Write $L_t:=L(1/t)$,
\[
U(t):=\int_0^\infty \cos(tu)h(u)\,du,
\qquad
V(t):=\int_0^\infty \sin(tu)h(u)\,du.
\]
Then, as $t\downarrow0$, the following estimates hold.
\begin{enumerate}
\item If $0<\alpha<1$, then
\[
1-U(t)\sim S(\alpha)t^\alpha L_t,
\qquad
V(t)\sim C(\alpha)t^\alpha L_t,
\]
where
\[
S(\alpha)=\frac{\pi/2}{\Gamma(\alpha)\sin(\pi\alpha/2)},
\qquad
C(\alpha)=\frac{\pi/2}{\Gamma(\alpha)\cos(\pi\alpha/2)}.
\]
\item If $\alpha=1$, then
\[
1-U(t)\sim \frac{\pi}{2}tL_t,
\qquad
2V(t)-V(2t)\sim 2(\log2)tL_t.
\]
More explicitly, with $J(x):=\int_1^x u^{-1}L(u)\,du$, the potentially larger slowly varying contribution to $V(t)$ is of order $tJ(1/t)$, and it cancels in $2V(t)-V(2t)$ because
\[
J(1/t)-J(1/(2t))\sim (\log2)L_t.
\]
\item If $1<\alpha<2$, then $\mu_1:=\int_0^\infty u h(u)\,du<\infty$, and
\[
1-U(t)\sim S(\alpha)t^\alpha L_t,
\qquad
V(t)-\mu_1t\sim C(\alpha)t^\alpha L_t.
\]
\item If $\alpha=2$, then $\mu_1:=\int_0^\infty u h(u)\,du<\infty$,
\[
V(t)-\mu_1t\sim -\frac{\pi}{2}t^2L_t,
\qquad
2V(t)-V(2t)\sim \pi t^2L_t.
\]
Moreover, if $R_2(x):=1+\int_1^x u^{-1}L(u)\,du$, then
\[
1-U(t)=O(t^2R_2(1/t)),
\qquad
 tR_2(1/t)=o(L_t).
\]
\end{enumerate}
\end{lem}

\begin{proof}
Let $X$ have density $h$ on $(0,\infty)$ and characteristic function
\[
\phi_X(t):=\E e^{\ii tX}=U(t)+\ii V(t).
\]
In Pitman's notation \cite{Pitman1968}, the tail sum and tail difference are
\[
H_P(x):=1-F(x)+F(-x),\qquad K_P(x):=1-F(x)-F(-x).
\]
Because $X\ge0$, for all large $x$ these are both equal to
\[
\overline H(x)=\PP(X>x)=\int_x^\infty h(u)\,du.
\]
Thus both $H_P$ and $K_P$ are regularly varying with index $-\alpha$. The additional monotonicity/one-sidedness hypothesis in Pitman's Theorem~7 is automatic here, since $K_2(x)=F(-x)=0$ for $x>0$.

For the real part, Pitman \cite[Theorem~1]{Pitman1968} gives
\[
1-U(t)\sim S(\alpha)\overline H(1/t),\qquad 0<\alpha<2,
\]
and Pitman \cite[Theorem~3]{Pitman1968} gives the endpoint estimate, when $\alpha=2$,
\[
1-U(t)\sim t^2\int_0^{1/t}x\overline H(x)\,dx.
\]
Equivalently, the latter endpoint may also be read from Pitman \cite[Theorem~6(iii), $n=0$]{Pitman1968}. Since
\[
\overline H(x)\sim x^{-\alpha}L(x),
\]
the first display gives the asserted real-part asymptotic for $0<\alpha<2$, and the second gives
\[
1-U(t)=O\{t^2R_2(1/t)\},
\qquad
R_2(x):=1+\int_1^x u^{-1}L(u)\,du,
\]
when $\alpha=2$.

For the sine transform, Pitman \cite[Theorem~7(i)]{Pitman1968} gives
\[
V(t)\sim C(\alpha)\overline H(1/t),\qquad 0<\alpha<1.
\]
Pitman \cite[Theorem~7, case $m=1$]{Pitman1968} gives
\[
V(t)-t\int_0^{1/t}\overline H(x)\,dx
\sim -\gamma\overline H(1/t),
\qquad \alpha=1,
\]
where $\gamma$ is Euler's constant. Finally, Pitman \cite[Theorem~8(i), $n=1$]{Pitman1968}, in the odd-moment-subtracted case, gives
\[
V(t)-\mu_1t\sim C(\alpha)\overline H(1/t),
\qquad 1<\alpha\le2,
\]
where the range in the printed statement of Theorem~8(i) is read with the standard correction $2n-1<m<2n+1$, noted in \cite[Lemma~6.2, after (6.13)]{PitmanPitman2016}, and where
\[
\mu_1=\int_0^\infty xh(x)\,dx=\int_0^\infty \overline H(x)\,dx<\infty.
\]
The finiteness of $\mu_1$ for $\alpha>1$ follows, for example, from Karamata's theorem for tails of regularly varying functions \cite[Theorem~1.5.11(ii)]{Bingham_Goldie_Teugels_1987}. The constants in Pitman's notation are
\[
S(\alpha)=\int_0^\infty \frac{\sin y}{y^\alpha}\,dy
=\frac{\pi/2}{\Gamma(\alpha)\sin(\pi\alpha/2)},
\]
and, for the sine transform after the appropriate odd-moment subtraction,
\[
C(\alpha)=\int_0^\infty \frac{\cos y}{y^\alpha}\,dy
=\frac{\pi/2}{\Gamma(\alpha)\cos(\pi\alpha/2)},
\]
with the second identity understood by the convergent subtracted-integral definition when $1<\alpha\le2$.

It remains only to spell out the two cancellations used later. For $\alpha=1$, set
\[
I(x):=\int_0^x\overline H(u)\,du,
\qquad
J(x):=\int_1^x u^{-1}L(u)\,du.
\]
The preceding Pitman expansion gives
\[
V(t)=tI(1/t)-\gamma\overline H(1/t)+o\{\overline H(1/t)\}.
\]
Since $\overline H(1/t)\sim tL_t$, where $L_t:=L(1/t)$,
\begin{align*}
2V(t)-V(2t)
&=2t\{I(1/t)-I(1/(2t))\}\notag\\
&\quad -2\gamma\overline H(1/t)+\gamma\overline H(1/(2t))+o(tL_t).
\end{align*}
The Euler-constant terms cancel because $\overline H(1/t)\sim tL_t$ and $\overline H(1/(2t))\sim2tL_t$. Moreover,
\[
I(x)-I(x/2)
=\int_{x/2}^x\overline H(u)\,du
\sim \int_{x/2}^x\frac{L(u)}{u}\,du
=J(x)-J(x/2),
\]
and the uniform convergence theorem for slowly varying functions \cite[Theorem~1.2.1]{Bingham_Goldie_Teugels_1987} gives
\[
J(x)-J(x/2)\sim(\log2)L(x).
\]
Therefore
\[
2V(t)-V(2t)\sim 2(\log2)tL_t.
\]

For $\alpha=2$, the same corrected $n=1$ case of Pitman \cite[Theorem~8(i)]{Pitman1968} gives
\[
V(t)-\mu_1t\sim -\frac{\pi}{2}t^2L_t,
\]
and so
\[
2V(t)-V(2t)
=2\{V(t)-\mu_1t\}-\{V(2t)-2\mu_1t\}
\sim \pi t^2L_t.
\]
It remains to justify the small auxiliary estimate $tR_2(1/t)=o(L_t)$. By Potter's theorem \cite[Theorem~1.5.6(i)]{Bingham_Goldie_Teugels_1987}, for every $\varepsilon\in(0,1)$ there are $x_0$ and $C_\varepsilon$ such that, whenever $x_0\le u\le x$,
\[
\frac{L(u)}{L(x)}\le C_\varepsilon\max\{(u/x)^\varepsilon,(u/x)^{-\varepsilon}\}.
\]
The contribution of $[1,x_0]$ to $\int_1^x u^{-1}L(u)\,du$ is $O(1)=O(x^\varepsilon L(x))$, because $x^\varepsilon L(x)\to\infty$. On $[x_0,x]$ Potter's bound gives
\[
\int_{x_0}^x \frac{L(u)}u\,du
\le
C_\varepsilon L(x)x^\varepsilon\int_{x_0}^x u^{-1-\varepsilon}\,du
=
O(x^\varepsilon L(x)).
\]
Thus $R_2(x)=O(x^\varepsilon L(x))$. Choosing $\varepsilon<1$ gives
\[
x^{-1}R_2(x)=o(L(x)),\qquad x\to\infty,
\]
which is the displayed $tR_2(1/t)=o(L_t)$.
\end{proof}

\subsection{Proof of Proposition~\ref{prop:optimal-local-contrast}.}

\begin{proof}
Let $K_H=[-H,H]^2$. For every admissible $g$,
\[
\mu_g
=
\int_{K_H} g(\tau)c_3^{\mathrm o}(\tau)\,d\tau
\le
\int_{K_H}|c_3^{\mathrm o}(\tau)|\,d\tau,
\]
because $|g|\le1$. Hence
\[
D_H\le \int_{K_H}|c_3^{\mathrm o}(\tau)|\,d\tau.
\]

For the reverse inequality, define
\[
g_H(\tau):=\operatorname{sgn}(c_3^{\mathrm o}(\tau))\,\1_{K_H}(\tau),
\]
with $\operatorname{sgn}(0)=0$. Since
\[
c_3^{\mathrm o}(-\tau)=-c_3^{\mathrm o}(\tau)
\quad\text{a.e.},
\]
we may modify $g_H$ on a null set, if desired, so that
\[
g_H(-\tau)=-g_H(\tau)
\]
everywhere. Then $g_H$ is admissible and
\[
\mu_{g_H}
=
\int_{K_H}|c_3^{\mathrm o}(\tau)|\,d\tau.
\]
Therefore
\[
D_H=\int_{K_H}|c_3^{\mathrm o}(\tau)|\,d\tau.
\]

Finally, since $K_H\uparrow\R^2$ and $|c_3^{\mathrm o}|\in L^1(\R^2)$,
monotone convergence gives
\[
D_H\uparrow \int_{\R^2}|c_3^{\mathrm o}(\tau)|\,d\tau
=
\|c_3^{\mathrm o}\|_1.
\]
The last claim is immediate.
\end{proof}

\begin{lem}[Diagonal imaginary part of the branching numerator]\label{lem:diagonal-imaginary-part}
Let $0<m<1$, let $h$ be a probability density on $(0,\infty)$, define
\[
U(t):=\int_0^\infty \cos(tu)h(u)\,du,
\qquad
V(t):=\int_0^\infty \sin(tu)h(u)\,du,
\]
and let $Q$ be the numerator term from Theorem~\ref{thm:branching-bispectrum}. Then $A(t):=\Imm Q(t,t)$ satisfies
\begin{align}
A(t)
&=2(1-m)(2V(t)-V(2t)) \notag\\
&\quad+2(1-2m)\{(1-U(t))(V(2t)-V(t))-(1-U(2t))V(t)\} \notag\\
&\quad+2m\{[(1-U(t))^2-V(t)^2]V(2t)-2(1-U(2t))(1-U(t))V(t)\}.
\label{eq:diagonal-imaginary-A}
\end{align}
\end{lem}

\begin{proof}
Since $\widehat h(t)=U(t)-\ii V(t)$ and $\widehat h(-t)=U(t)+\ii V(t)$, the diagonal numerator is
\[
Q(t,t)=\widehat h(-t)^2+\widehat h(2t)\{2\widehat h(-t)-2m\widehat h(-t)^2\}.
\]
Substituting the displayed forms of $\widehat h(t)$, $\widehat h(-t)$, and $\widehat h(2t)$, then collecting imaginary parts, gives \eqref{eq:diagonal-imaginary-A}.
\end{proof}

\subsection{Proof of Proposition~\ref{prop:nonzero-imag-bispectrum}.}
Let $B_{\mathrm{fac}}$ denote the reduced third factorial-cumulant bispectrum; this is the bispectrum denoted by $B$ in the statement. By Proposition~\ref{prop:L1-branching}, this bispectrum is well-defined because the Poisson branching genealogy is subcritical and the edge displacement law has a Lebesgue density. Let $B_{\mathrm{comp}}$ denote the Fourier--Stieltjes transform of the reduced complete, or ordinary, third cumulant measure. By Lemma~\ref{lem:ordinary-factorial-bispectrum},
\[
\Imm B_{\mathrm{fac}}=\Imm B_{\mathrm{comp}}.
\]
By Theorem~\ref{thm:branching-bispectrum},
\[
B_{\mathrm{comp}}(\omega_1,\omega_2)
=\lambda\,
\frac{1-m^2 Q(\omega_1,\omega_2)}
{|1-m\widehat h(\omega_1)|^2 |1-m\widehat h(\omega_2)|^2 |1-m\widehat h(-\omega_1-\omega_2)|^2},
\]
where
\[
Q(\omega_1,\omega_2)
:=\widehat h(-\omega_1)\widehat h(-\omega_2)
+\widehat h(\omega_1+\omega_2)
\{\widehat h(-\omega_1)+\widehat h(-\omega_2)-2m\widehat h(-\omega_1)\widehat h(-\omega_2)\}.
\]
On the diagonal this becomes
\[
Q(t,t)=\widehat h(-t)^2+\widehat h(2t)\{2\widehat h(-t)-2m\widehat h(-t)^2\}.
\]
Because the denominator is real and strictly positive,
\[
\Imm B_{\mathrm{fac}}(t,t)=\Imm B_{\mathrm{comp}}(t,t)
= -\lambda m^2\,
\frac{A(t)}{|1-m\widehat h(t)|^4 |1-m\widehat h(-2t)|^2},
\qquad
A(t):=\Imm Q(t,t).
\]
It is therefore enough to prove that $A(t)\neq0$ for some $t>0$.

Write
\[
U(t):=\int_0^\infty \cos(tu)h(u)\,du,
\qquad
V(t):=\int_0^\infty \sin(tu)h(u)\,du,
\]
so that $\widehat h(t)=U(t)-\ii V(t)$. By Lemma~\ref{lem:diagonal-imaginary-part}, $A(t)$ is given by \eqref{eq:diagonal-imaginary-A}.

Assume first that condition~(i) holds, and let $L_t:=L(1/t)$. We use Lemma~\ref{lem:small-frequency-abelian-estimates}. For $0<\alpha<1$,
\[
1-U(t)\sim S(\alpha)t^\alpha L_t,
\qquad
V(t)\sim C(\alpha)t^\alpha L_t.
\]
Thus $1-U(t)$, $1-U(2t)$, $V(t)$, and $V(2t)-V(t)$ are all $O(t^\alpha L_t)$. Since $t^\alpha L_t\to0$, the second line of \eqref{eq:diagonal-imaginary-A} is $O((t^\alpha L_t)^2)=o(t^\alpha L_t)$, and the third line is $O((t^\alpha L_t)^3)=o(t^\alpha L_t)$. Hence
\[
\lim_{t\downarrow0}\frac{A(t)}{t^\alpha L_t}
=2(1-m)(2-2^\alpha)C(\alpha)\neq0.
\]
For $1<\alpha<2$, one has
\[
1-U(t)\sim S(\alpha)t^\alpha L_t,
\qquad
V(t)-\mu_1t\sim C(\alpha)t^\alpha L_t,
\qquad
\mu_1:=\int_0^\infty u h(u)\,du.
\]
The linear term cancels in $2V(t)-V(2t)$. For the second line of \eqref{eq:diagonal-imaginary-A}, write $a_t:=1-U(t)$ and $b_t:=V(t)-\mu_1t$. Then $a_t=O(t^\alpha L_t)$ and $b_t=O(t^\alpha L_t)$, so
\[
a_t(V(2t)-V(t))-a_{2t}V(t)
=
\mu_1t(a_t-a_{2t})+a_t(b_{2t}-b_t)-a_{2t}b_t
=o(t^\alpha L_t).
\]
In the third line, the pure sine product is $V(t)^2V(2t)=O(t^3)=o(t^\alpha L_t)$, and all other products contain at least two factors among $a_t,a_{2t},b_t,b_{2t}$ or an extra factor of $t$; they are therefore also $o(t^\alpha L_t)$. Thus again
\[
\lim_{t\downarrow0}\frac{A(t)}{t^\alpha L_t}
=2(1-m)(2-2^\alpha)C(\alpha)\neq0.
\]
When $\alpha=1$,
\[
1-U(t)\sim \frac{\pi}{2}tL_t,
\qquad
2V(t)-V(2t)\sim 2(\log 2)tL_t.
\]
Equivalently, if $J(x)=\int_1^x u^{-1}L(u)\,du$, then $V(t)$ may contain a term of order $tJ(1/t)$, but this term cancels in $2V(t)-V(2t)$ up to
\[
J(1/t)-J(1/(2t))\sim(\log2)L_t.
\]
For the remaining lines of \eqref{eq:diagonal-imaginary-A}, set
\[
a_t:=1-U(t),\qquad v_t:=V(t),\qquad J_t:=J(1/t).
\]
The Abelian estimates above, together with slow variation, give
\[
a_t=O(tL_t),\qquad a_{2t}=O(tL_t),
\]
and the sine-transform bound
\[
v_t=O\{t(1+J_t)\},\qquad v_{2t}=O\{t(1+J_t)\},\qquad
v_{2t}-v_t=O\{t(1+J_t)\}.
\]
Potter's theorem for slowly varying functions \cite[Theorem~1.5.6(i), p.~25]{Bingham_Goldie_Teugels_1987} implies, after reducing $\varepsilon$ if necessary, that
\[
L_t+L_t^{-1}+1=O(t^{-\varepsilon}),\qquad t\downarrow0.
\]
It also gives $J_t=O(t^{-\varepsilon})$: with $x=1/t$ and a smaller exponent, $L(u)\le C u^{\varepsilon/2}$ for large $u$, while the bounded interval near $1$ contributes only $O(1)$, and therefore
\[
J(x)=\int_1^x u^{-1}L(u)\,du\le C_\varepsilon x^\varepsilon .
\]
Hence, for every sufficiently small $\varepsilon>0$,
\[
L_t+L_t^{-1}+1+J_t=O(t^{-\varepsilon}),\qquad t\downarrow0.
\]
Choosing $\varepsilon<1/2$, we obtain
\[
t(1+J_t)\to0,
\qquad
 t^2L_t(1+J_t)\to0,
\qquad
\frac{t^2(1+J_t)^3}{L_t}\to0 .
\]
Therefore the second line of \eqref{eq:diagonal-imaginary-A} satisfies
\[
a_t(v_{2t}-v_t)-a_{2t}v_t
=
O\{t^2L_t(1+J_t)\}
=
o(tL_t).
\]
For the third line, the terms involving $a_t$ are bounded by
\[
a_t^2v_{2t}+a_{2t}a_tv_t
=
O\{t^3L_t^2(1+J_t)\}
=
tL_t\,O\{t^2L_t(1+J_t)\}
=
o(tL_t),
\]
while the purely sine-transform product obeys
\[
v_t^2v_{2t}
=
O\{t^3(1+J_t)^3\}
=
tL_t\,O\!\left\{\frac{t^2(1+J_t)^3}{L_t}\right\}
=
o(tL_t).
\]
Hence both remaining lines of \eqref{eq:diagonal-imaginary-A} are $o(tL_t)$. Hence
\[
\lim_{t\downarrow0}\frac{A(t)}{tL_t}=4(1-m)\log2>0.
\]
Finally, for $\alpha=2$,
\[
V(t)-\mu_1t\sim -\frac{\pi}{2}t^2L_t,
\qquad
\mu_1:=\int_0^\infty u h(u)\,du.
\]
No leading asymptotic for $1-U(t)$ is needed: if
\[
R_2(x):=1+\int_1^x u^{-1}L(u)\,du,
\]
then Lemma~\ref{lem:small-frequency-abelian-estimates} gives $1-U(t)=O(t^2R_2(1/t))$ and $tR_2(1/t)=o(L_t)$. Also $V(t)=O(t)$, so
\[
(1-U(t))V(t)=O(t^3R_2(1/t))=o(t^2L_t),
\qquad
V(t)^2V(2t)=O(t^3)=o(t^2L_t).
\]
The remaining products in the last two lines of \eqref{eq:diagonal-imaginary-A} are smaller combinations of the same bounds, and hence are also $o(t^2L_t)$. The linear term in $V$ again cancels, and
\[
2V(t)-V(2t)\sim \pi t^2L_t,
\]
so
\[
\lim_{t\downarrow0}\frac{A(t)}{t^2L_t}=2\pi(1-m)>0.
\]
Thus in every regularly varying case there are arbitrarily small $t>0$ with $A(t)\neq0$.

Assume next that condition~(ii) holds. Let $X$ have density $h$. By Lemma~\ref{lem:uniform-mixture}, there are nonnegative $Z$ and $Y\sim \mathrm{Unif}(0,1)$, independent, such that $X\stackrel d=YZ$. The moment identities
\[
\E X=\frac12\E Z,
\qquad
\E X^2=\frac13\E Z^2
\]
show that $\E Z^2<\infty$ under the finite-second-moment assumption on $h$. Moreover, $h=a^{-1}\1_{(0,a)}$ exactly when $Z=a$ a.s. Under condition~(ii), $Z$ is therefore nondegenerate. With the usual continuous interpretation at $Z=0$,
\[
U(t)=\E\frac{\sin(tZ)}{tZ},
\qquad
V(t)=\E\frac{1-\cos(tZ)}{tZ},
\]
and
\[
2V(t)-V(2t)=\E\frac{(1-\cos(tZ))^2}{tZ}.
\]
Also
\[
\frac{V(t)}{t}\to \frac{1}{2}\E Z,
\qquad
\frac{1-U(t)}{t^2/2}\to \frac{1}{3}\E Z^2.
\]
If $\E Z^3<\infty$, then the mixture formulas and Taylor expansion yield
\[
1-U(t)=\frac{t^2}{6}\E Z^2+o(t^2),
\qquad
V(t)=\frac{t}{2}\E Z-\frac{t^3}{24}\E Z^3+o(t^3),
\]
and therefore
\[
2V(t)-V(2t)=\frac{t^3}{4}\E Z^3+o(t^3).
\]
Substituting these expansions into \eqref{eq:diagonal-imaginary-A} gives
\[
A(t)
=
\frac{t^3}{2}\Big((1-m)\E Z^3-(1-2m)\E Z\,\E Z^2-m(\E Z)^3\Big)
+o(t^3).
\]
Equivalently,
\begin{align*}
\lim_{t\downarrow0}\frac{A(t)}{t^3/2}
&=(1-m)\E Z^3-(1-2m)\E Z\,\E Z^2-m(\E Z)^3 \\
&=(1-m)(\E Z^3-\E Z\,\E Z^2)+m\E Z\{\E Z^2-(\E Z)^2\}.
\end{align*}
The first bracket is nonnegative because, for an independent copy $Z'$ of $Z$,
\[
\E Z^3-\E Z\,\E Z^2
=
\frac12\E[(Z-Z')^2(Z+Z')]\ge0.
\]
The second bracket is $\mathrm{Var}(Z)>0$ because $Z$ is nondegenerate, and $\E Z>0$. Hence the displayed limit is strictly positive.

If $\E Z^3=\infty$, Fatou's lemma applied to the nonnegative first-line contribution
\[
4(1-m)\E\left[\frac{(1-\cos(tZ))^2}{t^4Z}\right]
\]
to $A(t)/(t^3/2)$ gives an infinite lower limit. It remains to check that the other two lines of \eqref{eq:diagonal-imaginary-A}, after division by $t^3/2$, are bounded in absolute value. Put
\[
a_t:=1-U(t),\qquad v_t:=V(t).
\]
The mixture representation gives, for $s=1,2$,
\[
0\le a_{st}\le \frac{s^2t^2}{6}\E Z^2,
\qquad
0\le v_{st}\le \frac{st}{2}\E Z,
\]
using $1-\sin y/y\le y^2/6$ and $1-\cos y\le y^2/2$. Hence
\[
|a_t(v_{2t}-v_t)-a_{2t}v_t|
\le a_t(v_{2t}+v_t)+a_{2t}v_t=O(t^3),
\]
and
\[
|[a_t^2-v_t^2]v_{2t}-2a_{2t}a_tv_t|
\le a_t^2v_{2t}+v_t^2v_{2t}+2a_{2t}a_tv_t=O(t^3).
\]
Thus the remaining contributions to $A(t)/(t^3/2)$ are bounded, while the first-line contribution has infinite lower limit. Consequently $A(t)>0$ for all sufficiently small $t>0$ in this case as well. Therefore $A(t_0)\neq0$ for some $t_0>0$, and consequently $\Imm B(t_0,t_0)\neq0$. The final separation claim follows from Theorem~\ref{thm:real-bispectrum}.

\begin{rem}[Uniform-kernel exception and monotone sharpness]\label{rem:uniform-kernel-exception}
The exclusion in condition~(ii) of Proposition~\ref{prop:nonzero-imag-bispectrum} is not merely an artefact of the diagonal argument. For the uniform density $a^{-1}\1_{(0,a)}$, the numerator $Q$ in Theorem~\ref{thm:branching-bispectrum} is real for every $(\omega_1,\omega_2)$. Indeed,
\[
\widehat h(\omega)=e^{-\ii a\omega/2}\operatorname{sinc}(a\omega/2),
\]
with the continuous convention at zero. Writing $x=a\omega_1/2$, $y=a\omega_2/2$, $S=x+y$, and $s_z=\operatorname{sinc}(z)$, the $m$-dependent term in $Q$ is real and the remaining imaginary part is
\[
s_xs_y\sin S-s_xs_S\sin y-s_ys_S\sin x.
\]
For $xyS\ne0$, this equals
\[
\sin x\,\sin y\,\sin S
\left\{\frac1{xy}-\frac1{xS}-\frac1{yS}\right\}=0,
\]
and the remaining cases follow by continuity. Thus the imaginary bispectrum vanishes identically in the uniform case; on the diagonal this appears as the identity $A(t)=0$. This calculation proves the phase cancellation of the third-order transform; it does not by itself assert reversibility or non-reversibility of the unlabeled stationary point process. It shows that a nonzero imaginary bispectrum is a sufficient witness against reversible $L^1$ comparison nulls, not a necessary feature of every one-sided branching construction.

The monotone finite-second-moment exception is sharp. Indeed, by Lemma~\ref{lem:uniform-mixture}, if $X$ has a nonincreasing density $h$ on $(0,\infty)$, then
\[
X\stackrel d=YZ,
\qquad
Y\sim \mathrm{Unif}(0,1),
\]
with $Y$ independent of a nonnegative random variable $Z$. The small-frequency diagonal calculation in the proof of Proposition~\ref{prop:nonzero-imag-bispectrum} shows that, when $\E Z^3<\infty$, the leading coefficient is
\[
\Delta_m(Z)
=
(1-m)\{\E Z^3-\E Z\,\E Z^2\}
+
m\,\E Z\,\mathrm{Var}(Z).
\]
The first bracket is always nonnegative, since for an independent copy $Z'$,
\[
\E Z^3-\E Z\,\E Z^2
=
\frac12\E\big[(Z-Z')^2(Z+Z')\big]\ge0.
\]
Hence
\[
\Delta_m(Z)\ge m\,\E Z\,\mathrm{Var}(Z).
\]
Because $m>0$ and $\E Z>0$, this coefficient can vanish only if $\mathrm{Var}(Z)=0$, that is, only if $Z$ is deterministic. By Lemma~\ref{lem:uniform-mixture}, this deterministic-scale case is exactly
\[
h(x)=a^{-1}\mathbf 1_{(0,a)}(x)
\quad\text{a.e.}
\]
If $\E Z^3=\infty$, the same proof gives an infinite lower limit for the corresponding small-frequency coefficient, so degeneracy is again impossible. Thus, among monotone finite-second-moment kernels, the uniform law is the unique case not separated by the diagonal imaginary-bispectrum argument.
\end{rem}

\begin{rem}[Quantitative small-frequency lower bounds]\label{rem:small-frequency-lower-bounds}
The proof of Proposition~\ref{prop:nonzero-imag-bispectrum} gives explicit small-frequency lower bounds, not only non-vanishing. Throughout this remark $B$ denotes the reduced third factorial-cumulant bispectrum and $\lambda$ is the total intensity.

First assume
\[
\overline H(x):=\int_x^\infty h(u)\,du\sim x^{-\alpha}L(x),
\qquad 0<\alpha\le 2,
\]
and write $L_t:=L(1/t)$. Define
\[
\chi_\alpha:=
\begin{cases}
2(2-2^\alpha)C(\alpha), & 0<\alpha<2,\ \alpha\ne1,\\[3pt]
4\log 2, & \alpha=1,\\[3pt]
2\pi, & \alpha=2,
\end{cases}
\]
where
\[
C(\alpha):=\frac{\pi/2}{\Gamma(\alpha)\cos(\pi\alpha/2)}
\qquad (0<\alpha<2,\ \alpha\ne1).
\]
Then $\chi_\alpha>0$, and
\[
\lim_{t\downarrow0}
\frac{|\Imm B(t,t)|}{t^\alpha L_t}
=
\frac{\lambda m^2}{(1-m)^5}\,\chi_\alpha .
\]
Consequently, for every $\eta\in(0,1)$, there exists $t_\eta>0$ such that, for $0<t<t_\eta$,
\[
|\Imm B(t,t)|
\ge
(1-\eta)\frac{\lambda m^2}{(1-m)^5}\,
\chi_\alpha\,t^\alpha L(1/t).
\]

Second assume that $h$ is nonincreasing on $(0,\infty)$ and has finite second moment. Let $X\sim h$, and use the scale-mixture representation from Lemma~\ref{lem:uniform-mixture},
\[
X\stackrel d=YZ,
\qquad
Y\sim\mathrm{Unif}(0,1),
\]
with $Y$ independent of $Z$. If $\E Z^3<\infty$, set
\[
\Delta_m(Z)
:=
(1-m)\{\E Z^3-\E Z\,\E Z^2\}
+
m\,\E Z\,\mathrm{Var}(Z).
\]
Then
\[
\lim_{t\downarrow0}
\frac{|\Imm B(t,t)|}{t^3}
=
\frac{\lambda m^2}{2(1-m)^6}\,\Delta_m(Z).
\]
Moreover,
\[
\Delta_m(Z)\ge m\,\E Z\,\mathrm{Var}(Z),
\]
because
\[
\E Z^3-\E Z\,\E Z^2
=
\frac12\E[(Z-Z')^2(Z+Z')]\ge0,
\]
where $Z'$ is an independent copy of $Z$. Thus any class separated from the uniform family in the sense that
\[
\E Z\,\mathrm{Var}(Z)\ge \delta>0
\]
satisfies, for all sufficiently small $t>0$,
\[
|\Imm B(t,t)|
\ge
\frac{\lambda m^3\delta}{4(1-m)^6}\,t^3 .
\]
If $\E Z^3=\infty$, then the same proof gives the stronger conclusion
\[
\liminf_{t\downarrow0}
\frac{|\Imm B(t,t)|}{t^3}
=
\infty .
\]
The exceptional uniform kernel is exactly the case $Z=a$ a.s.; equivalently, $\mathrm{Var}(Z)=0$.

Equivalently, if one uses the second-order-normalized population signal
\[
\mathsf S(t):=
\frac{|\Imm B(t,t)|}{\{\Gamma(t)^2\Gamma(2t)\}^{1/2}},
\]
then $\Gamma(t)\to \lambda/(1-m)^2$ as $t\downarrow0$. Hence, in the regularly varying case,
\[
\lim_{t\downarrow0}
\frac{\mathsf S(t)}{t^\alpha L(1/t)}
=
\frac{m^2}{\sqrt{\lambda}(1-m)^2}\,\chi_\alpha,
\]
while in the finite-third-moment monotone case,
\[
\lim_{t\downarrow0}
\frac{\mathsf S(t)}{t^3}
=
\frac{m^2}{2\sqrt{\lambda}(1-m)^3}\,\Delta_m(Z).
\]
\end{rem}

\subsection{Proof of Remark~\ref{rem:branching-ratio-envelope}.}
By Lemma~\ref{lem:pcp-cumulant-functional},
\[
\kappa_3^{\mathrm{red}}
=
\nu\,\E\sum_{x_0,x_1,x_2\in C}^{\neq}
\delta_{x_1-x_0}\otimes\delta_{x_2-x_0}.
\]
Hence $\kappa_3^{\mathrm{red}}$ is a positive finite measure with total mass
\[
\kappa_3^{\mathrm{red}}(\R^2)=\nu\,\E[(M)_3],
\qquad M:=C(\R).
\]
If it has density $c_3$, then
\[
|B_{\mathrm{fac}}(\omega_1,\omega_2)|
\le
\iint_{\R^2}c_3(\tau_1,\tau_2)\,d\tau_1d\tau_2
=
\nu\,\E[(M)_3],
\]
and the same bound holds for the imaginary part.

For Poisson offspring with mean $m$, the total progeny generating function satisfies
\[
F(z):=\E[z^M]=z\exp(m(F(z)-1)).
\]
Differentiating three times at $z=1$ gives
\[
F^{(3)}(1)=\E[(M)_3]
=
\frac{m^2(2m^2-8m+9)}{(1-m)^5}.
\]
Substituting $\nu=\lambda(1-m)$ yields the fixed-intensity envelope
\[
\nu\,\E[(M)_3]
=
\lambda\,\frac{m^2(2m^2-8m+9)}{(1-m)^4}.
\]
Finally, under the scale family $h_\beta(t)=\beta h(\beta t)$, all cluster edge delays are divided by $\beta$. Therefore the two-dimensional reduced third-cumulant density rescales as
\[
c_3^{(\beta)}(\tau_1,\tau_2)
=\beta^2c_3^{(1)}(\beta\tau_1,\beta\tau_2),
\]
and Fourier transformation gives
\[
B^{(\beta)}(\omega_1,\omega_2)
=B^{(1)}(\omega_1/\beta,\omega_2/\beta).
\]
\end{document}